\documentclass{amsart}
\usepackage{graphicx}
\usepackage{graphicx,amsmath,amssymb}
\newtheorem{thm}{Theorem}[section]
\newtheorem{cor}[thm]{Corollary}
\newtheorem{lem}[thm]{Lemma}
\newtheorem{prop}[thm]{Proposition}
\theoremstyle{definition}

\theoremstyle{remark}
\newtheorem{rem}[thm]{Remark}

\numberwithin{equation}{section}

\begin{document}

\title[Sum rules ...]{Sum rules and a second order Feynman-Hellmann theorem for abstract operators with applications}%
\author{Joachim Stubbe}
\address{Joachim Stubbe, EPFL SB-DO, Station 3, CH-1015 Lausanne, Switzerland}%
\email{Joachim.Stubbe@epfl.ch}%
\thanks{}%
\subjclass{subjects}%

\date{May, 23th, 2026}
\begin{abstract}
We discuss the role of the Feynman-Hellmann theorem for abstract one-parameter families of Hamiltonians in sum rules and trace identities of Harrell and the author and its application to spectral theory. In particular, we derive a sum rule for the second derivative of eigenvalues of a one-parameter family of Hamiltonians extending thereby concepts of second order perturbation theory. We present applications to semiclassical eigenvalue bounds of Schr\"{o}dinger operators as Lieb-Thirring inequalities, zeros of Bessel functions, eigenvalue inequalities for sums of matrices and trace inequalities.
\end{abstract}
\maketitle
\section{Introduction}
Sum rules and the Feynman-Hellmann theorem are important concepts as well as powerful tools since the early days of quantum mechanics for computing atomic and molecular structure and its interaction with radiation. In the past decades both tools have been also repeatedly used in spectral theory and spectral geometry. While, roughly spoken, sum rules relate eigenvalues/energies of a fixed quantum system to other quantities as expectation values, the Feynman-Hellmann theorem relates the change in a system's energy to the expectation value of the first derivative of the Hamiltonian with respect to a parameter. The main objective of this paper is twofold: First, to make use of the Feynman-Hellman theorem for an abstract operator depending on a parameter in sum rules and, second, to present an abstract sum rule for such operators including also the expectation value of the second derivative of the operator with respect to the parameter. We present applications to Schr\"{o}dinger and Laplace-Beltrami operators, zeros of special functions and to eigenvalues of matrices proving new sharp estimates for eigenvalues.
 
Quantum mechanical sum rules relate spectral quantities like eigenvalues or differences of eigenvalues (transition energies), and expectation values of a linear operator describing a quantum system to other physical properties of this system. The starting point is the Thomas-Reiche-Kuhn (TRK) sum rule, first presented hundred years ago, relating dipole transition matrix elements - also called oscillator strengths - to the energy differences of the corresponding bound states. More precisely, the TRK sum rule states that the sum of the oscillator strengths of all possible transitions from a given bound state equals the number of particles in this state (\cite{Tho1925}, \cite{Ku1925}, \cite{ReiTho1925}). The Thomas-Reiche-Kuhn (TRK) sum rule was extended by H. Bethe in 1930 (\cite{Bet1930}). These sum rules can be derived from commutator algebra for the underlying Hamiltonian $H$. For a summary and applications we refer to the textbook by Bethe and Salpeter  (\cite{BetSal1957}, pp. 255-281, and pp.357-359) and to the review article (\cite{Ino1971}).

The technique of sum rules and commutation relations was introduced into spectral theory in the nineties of the past century beginning with the work of Harrell \cite{Har1993}, Harrell and Michel \cite{HarMi1994} with applications to eigenvalues of Laplace and Schr\"{o}dinger operators and Arai \cite{Arai1992} proving an abstract version of the TRK sum rule applying it to Sturm-Liouville operators with well-known special functions as eigensolutions. A new quadratic sum rule was shown by Harrell and the author in a seminal paper on abstract Schr\"{o}dinger operators \cite{HarStu1997} and later extended by them to trace identities and sum rules for general abstract self-adjoint operators $H$ acting on a Hilbert space \cite{HarStu2011}. Applied to Schr\"{o}dinger operators and Laplace-Beltrami this quadratic sum rule yields, for example, sharp estimates for the bottom of the spectrum either in form of universal inequalities between eigenvalues and for eigenvalue gaps or phase space bounds and has a large potential for further applications in spectral geometry. Indeed, recently Provenzano and the author extended this sum rule to compact homogeneous irreducible Riemannian manifolds and found a remarkable universal property of Laplacian spectra on compact, rank one symmetric spaces \cite{ProStu2025}. In the present paper we will apply the quadratic sum rule to the zeros of special functions and to hermitian matrices.

An appropriate and simple abstract setting unifying the presentation and derivation of the aforementioned sum rules is the following: Consider a self-adjoint operator $H$ acting on a Hilbert space $(\mathcal{H},\langle,\rangle)$ with eigenvalues $\lambda_j$, labeled in increasing order, and eigenvectors $u_j$ forming an orthonormal basis of the underlying Hilbert space. The case where $H$ also has continuous spectrum is settled by Harrell and the author in \cite{HarStu2010} allowing therefore the application of the sum rules described below to operators in quantum physics. Let $G$ be another self-adjoint operator such that the quantities $[H,G]=HG-GH$ and $[G, [H,G]]$ are well defined on the eigenvectors $u_j$ (which is typically true if $G$ is a bounded operator, otherwise see \cite{HarStu2011},\cite{HarStu2010} for the precise hypotheses). Then by the completeness of the eigenvectors $u_j$.
\begin{equation}\label{TRK-abstract}
	\langle[G, [H,G]] u_j,u_j\rangle = 2\sum_{\lambda_k}(\lambda_k-\lambda_j)|\langle G u_j,u_k\rangle |^2.
\end{equation}
This is the abstract version of the TRK sum rule. When $H=-\Delta +V(x)$ $V$ real-valued, $\mathcal{H}=L^2$, is a Schr\"{o}dinger operator defined on $\mathbb{R}^d$ or on a subdomain with Dirichlet boundary conditions or on an immersed manifold one may choose $G$ as the multiplication operator by a coordinate to get the known TRK sum rule
\begin{equation}\label{TRK-Laplace-1}
d = \sum_{\lambda_k}(\lambda_k-\lambda_j)\sum_{a=1}^{d}|\langle x_a u_j,u_k\rangle |^2,
\end{equation}
or choosing the multiplication operator $G=w\cdot x$ in $\mathbb{R}^d$ 
\begin{equation}\label{TRK-Laplace-2}
	w\cdot w = \sum_{\lambda_k}(\lambda_k-\lambda_j)|\langle (w\cdot x) u_j,u_k\rangle |^2.
\end{equation}
As shown in \cite{HarStu2011} for $G$ not necessarily self-adjoint with adjoint $G^*$ the abstract TRK sum rule \eqref{TRK-abstract} becomes
\begin{equation}\label{TRK-abstract-general}
	\begin{split}
	&\langle[G^*, [H,G]] u_j,u_j\rangle+\langle[G, [H,G^*]] u_j,u_j\rangle \\ &=2\sum_{\lambda_k}(\lambda_k-\lambda_j)\bigg(|\langle G u_j,u_k\rangle |^2+|\langle G^* u_j,u_k\rangle |^2\bigg).
	\end{split}
\end{equation}
The advantage of the sum rule \eqref{TRK-abstract-general} in the application to Schr\"{o}dinger operators on unbounded domains is that we may choose bounded operators $G$ as, for example, the unitary multiplication operator $G(x)=e^{iq\cdot x}$ so that no additional hypotheses on operator domains are necessary. This choice yields Bethe's sum rule (\cite{Bet1930}, see also \cite{Wang1999}). Since $V(x)$ is real we may choose the eigenfunctions to be real functions (as for real symmetric matrices we may choose eigenvectors with real entries) we have $|\langle e^{iq\cdot x} u_j,u_k\rangle |^2=|\langle e^{-iq\cdot x} u_j,u_k\rangle |^2$ and therefore
\begin{equation}\label{Bethe-sum-rule}
	q\cdot q = \sum_{\lambda_k}(\lambda_k-\lambda_j)|\langle e^{iq\cdot x} u_j,u_k\rangle |^2.
\end{equation}
In the abstract framework introduced before the quadratic sum rule can be stated as follows (see \cite{HarStu2011}). Let $J$ be a subset of the spectrum of $H$. For all $z$ real we have
\begin{equation}\label{HS-sumrule-discrete}
	\begin{split}
		&\frac1{2}\sum_{\lambda_j\in J}  (z-\lambda_j)^2\,\big(\langle[G^*,[H,G]]\phi_j,\phi_j\rangle+\langle[G,[H,G^*]]\phi_j,\phi_j\rangle\big)\\
		&-\sum_{\lambda_j\in
			J}(z-\lambda_j)\,\big(\langle[H,G]\phi_j,[H,G]\phi_j\rangle+\langle[H,G^*]\phi_j,[H,G^*]\phi_j\rangle\big)\\
		&=\\
		&\sum_{\lambda_j\in J}\sum_{\lambda_k\notin J}
		(z-\lambda_j)(z-\lambda_k)(\lambda_k-\lambda_j)\big(|\langle G\phi_j,\phi_k\rangle|^2+|\langle G^*\phi_j,\phi_k\rangle|^2\big).\\
	\end{split}
\end{equation}
 We remind that \eqref{HS-sumrule-discrete} is completely algbraic and no variational characterization of eigenvalues is needed. Comparing the coefficients of $z^2,z^1,z^0$ we obtain $3$ different sum rules for each eigenvalue $\lambda_j$ where the coefficient of $z^2$ is the abstract TRK sum rule \eqref{TRK-abstract-general}. The importance of the quadratic sum rule \eqref{HS-sumrule-discrete} has been shown in a series of applications to Dirichlet Laplacians, Schr\"{o}dinger operators and Laplace-Beltrami operators on compact manifolds: sharp estimates for eigenvalue moments of the bottom of the spectrum choosing $J=\{\lambda_1,\cdots \lambda_n\}$, bounds for eigenvalue gaps if $\lambda_{n+1}>\lambda_{n}$  (\cite{HarStu1997},\cite{HarStu2011}) and bounds on Riesz-means for these operators in terms of geometric properties via semiclassical limits (see e.g. \cite{HarStu1997},\cite{HarHer2008},\cite{HarStu2011}, \cite{ProStu2025}). Indeed, the key observation is that for $z\in[\lambda_{n},\lambda_{n+1}]$ the r.h.s. of \eqref{HS-sumrule-discrete} is non-positive and bounded by the quadratic polynomial (in $z$) given by
 \begin{equation}\label{HS-sumrule-discrete-upper bound}
 \frac{1}{2}\,(z-\lambda_{n})(z-\lambda_{n+1})\sum_{j=1}^n\langle[G^*, [H,G]] u_j,u_j\rangle+\langle[G, [H,G^*]] u_j,u_j\rangle.
 \end{equation}
 
 The Feynman-Hellmann theorem also presents one of the important achievements in early quantum mechanics. It was very likely first proven by G\"{u}ttinger \cite{Güt1932}  although the relation was postulated before by various authors (see also the historical review \cite{Wal2005}) then independently by Hellmann \cite{Hel1933} and Feynman \cite{Fey1939}. Its applications in atomic and molecular physics are widely discussed in the text book by Thirring \cite{Thi2002}, see also the brief overview \cite{PolMur2018}. In other fields of applied mathematics we find applications to zeros of special functions \cite{IsmZha1988} and to (not necessarily symmetric) matrices (see e.g. Theorem 5 in \cite{Lan1964})
 We proceed the discussion of the Feynman-Hellmann theorem in the abstract setting introduced above for sum rules.
 
 Let $H=H(\tau)=H_{\tau}$, $\tau$ real,  be a one parameter family of self-adjoint operator on a  Hilbert space $(\mathcal{H},\langle,\rangle)$ with common domain $\mathcal{D}_H$ having - for simplicity - purely discrete spectrum consisting of eigenvalues $\lambda_j=\lambda_j(\tau)=\lambda_{j,\tau}$, $j=1,2,\cdots$, numbered in increasing order with corresponding eigenvectors $u_j=u_j(\tau)=u_{j,\tau}$ forming an orthonormal basis of the underlying Hilbert space. For any $\delta,\tau$  and any $j$ the integral Feynman Hellmann theorem is a direct consequence of the self-adjointness and may be stated as follows:
 
 \begin{equation}\label{Feynman-Hellman-integral-version}
 	\langle  u_{j,\tau+\delta},(H_{\tau+\delta}-H_{\tau})u_{j,\tau}\rangle=(\lambda_{j,\tau+\delta}-\lambda_{j,\tau})\langle u_{j,\tau+\delta},u_{j,\tau}\rangle.
 \end{equation}
 
 Now the standard version of Feynman-Hellmann theorem can be easily proven dividing the integral Feynman Hellmann theorem by $\delta\neq 0$ and as in \cite{IsmZha1988} assuming the existence of the following limits:
 \begin{equation*}
 	\begin{split}
 	&\underset{\delta\to 0}{\lim} \langle u_{j,\tau+\delta},u_{j,\tau}\rangle=\langle u_{j,\tau},u_{j,\tau}\rangle=1, \\
 	 &\langle u_{j,\tau},\dot{H}_{\tau} u_{j,\tau}\rangle:=\underset{\delta\to 0}{\lim} \langle  
 		 \frac{H_{\tau+\delta}-H_{\tau}}{\delta} \, u_{j,\tau+\delta},u_{j,\tau}\rangle.\\
 		 \end{split} 
 \end{equation*}
 Then $\lambda_{j,\tau}$ is differentiable in $\tau$ and 
 \begin{equation}\label{Feynman-Hellmann theorem-1}
 	\dot{\lambda}_{j,\tau}=\langle u_{j,\tau},\dot{H}_{\tau} u_{j,\tau}\rangle.
 \end{equation}
 In particular, the Feynman-Hellmann theorem holds for all non-degenerate eigenvalues of $H$ (see e.g.\cite{Thi2002}, p.149). The conditions of the Feynman-Hellmann theorem are also met when the mapping $\tau\to H(\tau) $ is sufficiently regular. Since in most of the applications $H$ is of the form $H=H_0+f(\tau)H'$ for some analytic function $f$ we shall work with this strong hypothesis. Then the Kato-Rellich theorem states that for any non-degenerate eigenvalue $\lambda_j$ of $H(\tau_0)$ there is a neighborhood of $\tau_0$ such that $\lambda_j$ is analytic and $u_j$ is an analytic normalized eigenvector (see e.g. \cite{Kato1980}, pp. 395). This extends also to non-degenerate eigenvalues when eigenvalues are correctly relabeled (see e.g. \cite{Thi2002}) when the perturbation $H'$ is bounded relative to $H_0$ (\cite{Thi2002} theorem 3.5.13, p. 147). It is a well-known fact that this relabeling is potentially in conflict with our labeling of the eigenvalues in increasing order. Consider, for example, the $2x2$-matrix
 \begin{equation}
 	H(\tau)=\left(\begin{array}{cc}0& \tau\\ \tau & 0\end{array}\right)
 \end{equation}
 which is an analytic operator family with analytic eigenvalues $-\tau, \tau$ and corresponding eigenvectors. However, labeled in increasing order, $\lambda_1=\min(-\tau, \tau)0=-|\tau|$ and $\lambda_2=|\tau|$ are not differentiable at $\tau=0$ and their corresponding eigenvectors are even discontinuous. When degenerate eigenvalues appear at some $\tau_0$ and $\tau_0$ is a crossing point we consider the piecewise analytic eigenvalues, relabel them if necessary at $\tau_0$ in increasing order for the next interval (see \cite{Kato1980}). Much weaker regularity conditions are needed when $H$ is a symmetric matrix (see e.g. Rellich's theorem, theorem 6.8. in \cite{Kato1980}, \cite{KriMic2003}).
 
 The application of the Feynman-Hellmann theorem in sum rules of the form \eqref{HS-sumrule-discrete} has been introduced by the author in his alternative proofs of the Golden-Thompson inequality and sharp Lieb-Thirring inequalities for Schr\"{o}dinger operators \cite{Stu2010}, see also \cite{FLW2022}, chapter 7, for the Lieb-Thirring inequality).  We give a new version of the proof and further applications of the Feynman-Hellmann theorem in sum rules in the present paper (see Section 2). The key ingredient of the proofs is that the sum rule yields a monotonicity property of certain eigenvalue moments for the bottom or the full discrete spectrum with respect to the semiclassical parameter of Schr\"{o}dinger operators. Our first result is the use of the Feynman-Hellmann theorem in the abstract quadratic sum rule \eqref{HS-sumrule-discrete} for  a the one-parameter family of self-adjoint operators $H=H_{\tau}$.
 
  \begin{thm} \label{thm-Feynman-Hellmann theorem in quadratic sum rules}
  	 $H=H_{\tau}$, $\tau$ real,  be an analytic one parameter family of self-adjoint operators on a  Hilbert space $(\mathcal{H},\langle,\rangle)$ with common domain $\mathcal{D}_H$ having purely discrete spectrum consisting of eigenvalues $\lambda_j=\lambda_{j,\tau}$, $j=1,2,\cdots$, analytic in $\tau$, with corresponding eigenvectors $u_j=u_{j,\tau}$ forming a complete orthonormal basis of the underlying Hilbert space. Let $G_{\alpha}$, $\alpha=1,\dots d$ be a family of linear operators with adjoint $G_{\alpha}^*$ such that all first and second commutators with $H$ are well-defined. Suppose in addition that
  \begin{equation}\label{eq-energy-bounds-on-commutators-1}
  	\sum_{\alpha}\frac1{2}\langle[G_{\alpha}^*,[H,G_{\alpha}]]\phi_j,\phi_j\rangle+\langle[G_{\alpha},[H,G_{\alpha}^*]]\phi_j,\phi_j\rangle=1
  \end{equation}
  and
   \begin{equation}\label{eq-energy-bounds-on-commutators-2}
  	\sum_{\alpha}\langle[H,G_{\alpha}]\phi_j,[H,G_{\alpha}]\phi_j\rangle+\langle[H,G_{\alpha}^*]\phi_j,[H,G_{\alpha}^*]\phi_j\rangle \leq \eta(\tau)\lambda_j+ \theta(\tau)\dot{\lambda}_j
  \end{equation}
  for all $j$ where $\eta,\theta$ are continuous functions such that $1+\eta>0$ and $\theta\neq 0$. Let $A(\tau), B(\tau)$ be positive solutions of the ordinary differential equations
   \begin{equation}\label{eq-energy-bounds-on-commutators-ODEs}
  \dot{A}=-\frac{1+\eta}{\theta}\,A,\quad \dot{B}=-\frac{\eta}{\theta}\,B.
  \end{equation} 
  Then for all $z$ the quantity
  \begin{equation}\label{Riesz-mean-monotonicity}
  	A(\tau)^{-2}\sum_j(zB(\tau)-\lambda_j(\tau))_{+}^2
  \end{equation}
  is non-increasing if $\theta>0$ and  non-decreasing if $\theta<0$ where $x_{+}$ denotes the positive part of the quantity $x$.
  \end{thm}
  	
 A natural question is to ask about higher derivatives of eigenvalues. There are only few results in the literature about global concavity/convexity properties of eigenvalues beyond (second order) perturbation theory except for finite dimensional matrices. For example, a well known result for hermitian matrices is that the sum over the lowest eigenvalues is a concave function on the space of hermitian matrices (see e,g.\cite{LiebSie1991}, \cite{Kos2000}). Similarly, in quantum mechanics the lowest eigenvalue of a Schr\"{o}dinger operator $H=H_0+\tau H_1$ given by the variational principle is a concave function of $\tau$ (\cite{Thi2002}). This easily extends to the sum of the lowest eigenvalues and to Schr\"{o}dinger operators satisfying $\ddot{H}\leq 0$ in the sense of operators (see also \cite{Kos2000}).  A non-perturbative and optimal result with applications to atomic spectra proving convextiy/concavity properties of lowest eigenvalues for given angular momentum for a spherically symmetric potential $V$ with respect to the corresponding eigenvalues of the hydrogen atom when the Laplacian of $V$ outside the origin has a given sign can be found in \cite{MarStu1991}.
 
 As second order perturbation theory indicates, we cannot expect an analogous simple identity as \eqref{Feynman-Hellmann theorem-1} for $\ddot{\lambda}_{j,\tau}$, that is $	\ddot{\lambda}_{j,\tau}=\langle u_{j,\tau},\ddot{H}_{\tau} u_{j,\tau}\rangle$ except if $H_{\tau}$ commutes with $\dot{H}_{\tau}$. Indeed taking the derivative in the Feynman-Hellman theorem \eqref{Feynman-Hellmann theorem-1} we get the second derivative formula
 
 \begin{equation}\label{Feynman-Hellmann theorem-2nd-derivative}
 	\ddot{\lambda}_{j,\tau}=\langle u_{j,\tau},\ddot{H}_{\tau} u_{j,\tau}\rangle+2\langle \dot{u}_{j,\tau},\dot{H}_{\tau} u_{j,\tau}\rangle
 \end{equation}
 which we write for convenience as
 \begin{equation}\label{Feynman-Hellmann theorem-2nd-derivative-short}
 	\ddot{\lambda}_{j}=\langle u_{j},\ddot{H} u_{j}\rangle+2\langle \dot{u}_{j},\dot{H} u_{j}\rangle.
 \end{equation}
 This equation is the key to the main result of this paper presenting a second order Feynman-Hellmann theorem as a sum rule expanding the scalar product via the completeness relation as we do in second order perturbation theory. 
 
 \begin{thm} \label{thm-second order Feynman-Hellmann theorem} Let $H=H_{\tau}$, $\tau$ real,  be an analytic one parameter family of self-adjoint operators on a  Hilbert space $(\mathcal{H},\langle,\rangle)$ with common domain $\mathcal{D}_H$ having purely discrete spectrum consisting of eigenvalues $\lambda_j=\lambda_{j,\tau}$, $j=1,2,\cdots$, analytic in $\tau$, with corresponding eigenvectors $u_j=u_{j,\tau}$ forming a complete orthonormal basis of the underlying Hilbert space. We also suppose that the domains of the operator derivatives $\dot{H}$ and $\ddot{H}$ contain $\mathcal{D}_H$ and map into $\mathcal{D}_H$.  For any function $f:\mathbb{R}\longrightarrow\mathbb{R}$ and any interval $I$ for which eigenvalues and eigenfunctions are analytic in $\tau$:
 	\begin{equation}\label{Feynman-Hellmann-sum-rule-general}
 		\sum_{\lambda_j} f(\lambda_j)(\ddot{\lambda}_j -\langle \ddot{H}u_j,u_j\rangle)=\underset{\lambda_j\neq\lambda_k}{\sum\sum}\frac{f(\lambda_k)-f(\lambda_j)}{\lambda_k-\lambda_j}\big|\langle u_k,\dot{H}u_j\rangle\big|^2
 	\end{equation}
 	provided all sums which are taken over the full spectrum are finite. 
 	
 	When $f$ is a $C^1$ function such that $f'$ is concave then for any subset $J$ of the spectrum
 	\begin{equation}\label{Feynman-Hellmann-sum-rule-general-ineq-f-prime-concave}
 		\begin{split}
 			\sum_{\lambda_j\in J} f(\lambda_j)\ddot{\lambda}_j +f'(\lambda_j)\dot{\lambda}_j^2&\geq  \sum_{\lambda_j\in J} f(\lambda_j)\langle \ddot{H}u_j,u_j\rangle+ f'(\lambda_j)\langle \dot{H}u_j,\dot{H}u_j\rangle\\
 			&-\sum_{\lambda_j\in J}\sum_{\lambda_k\notin J}\frac{2f(\lambda_j)+f'(\lambda_j)(\lambda_k-\lambda_j)}{\lambda_k-\lambda_j}\big|\langle u_k,\dot{H}u_j\rangle\big|^2\\
 		\end{split}
 	\end{equation}
 	with equality when $f'$ is a linear function. In particular, for all $z=z(\tau)$
 	\begin{equation}\label{Feynman-Hellmann-sum-rule-quadratic}
 		\begin{split}
 			&\sum_{\lambda_j\in J} (z-\lambda_j)^2(\langle \ddot{H}u_j,u_j\rangle-\ddot{\lambda}_j) -2(z-\lambda_j)(\langle \dot{H}u_j,\dot{H}u_j\rangle-\dot{\lambda}_j^2)\\
 			&=2\sum_{\lambda_j\in J}\sum_{\lambda_k\notin J}\frac{(z-\lambda_j)(z-\lambda_k)}{\lambda_k-\lambda_j}\big|\langle \dot{H}u_j,u_k\rangle\big|^2.\\
 		\end{split}
 	\end{equation}
 \end{thm}
 We would like to add the following comments on our main result.
 \begin{rem} The proof of the theorem will essentially follow the proof of the quadratic sum rule \eqref{HS-sumrule-discrete} given in \cite{HarStu2010} or \cite{HarStu2011} but here based on a gap formula for $\dot{H}$ and not the commutator $[G,H]$ as for \eqref{HS-sumrule-discrete}. At least formally, one may view $\dot{H}$ as commutator of the parameter differentiation $d/d\tau$ with $H$ when acting on eigenfunctions. Therefore the complement $J^c$ of $J$ may also contain continuous spectrum and the sum in $J^c$ will be replaced by the spectral integral as in \cite{HarStu2010}. 
 \end{rem}
 \begin{rem} In some applications with non-degenerate eigenvalues we wish to label the eigenvalues in increasing order so that we have to relabel the eigenvalues at crossing points according to the discussion above. When $J$ is a proper subset of the spectrum the theorem only holds as long as no eigenvalue is leaving or entering $J$. However, we may stop at such values of $\tau$ and restart with the new number of eigenvalues in $J$ (compare also \cite{Stu2010}).
 \end{rem}
 \begin{rem} For any primitive $F$ of $f$ the l.h.s. of inequality \eqref{Feynman-Hellmann-sum-rule-general-ineq-f-prime-concave} is equal to $\displaystyle  \sum_{\lambda_j\in J} \ddot{F}(\lambda_j)$ while the expectation values on the right hand side do not correspond to the expectation values of $\ddot{F}(H)$ since the chain rule for matrices contains commutators of $H$ and $\dot{H}$. 
 \end{rem}
 \begin{rem} In applications it may be also advantageous to consider the shifted operators $H_z(\tau):=H(\tau)-z(\tau)$ where $z(\tau)$ is an analytic function of $\tau$, with emphasis on the choice $z=\lambda_{m}$ for some eigenvalue of $H$. the eigenvalues are then $\lambda_{z,j}=\lambda_{j}-z$ and therefore the expressions $\langle \ddot{H}u_j,u_j\rangle-\ddot{\lambda}_j)$ and $\langle \dot{H}u_j,\dot{H}u_j\rangle-\dot{\lambda}_j^2$ remain invariant under this transformation.
 \end{rem}
 \begin{rem} We will also show that the quadratic second order Feynman-Hellmann relation \eqref{Feynman-Hellmann-sum-rule-quadratic} implies the quadratic sum rule \eqref{HS-sumrule-discrete} for self-adjoint $G$ and more generally \eqref{Feynman-Hellmann-sum-rule-general} implies the abstract sum rule stated in \cite{HarStu2010}. Theorem \ref{thm-2nd-order-Feynman-Hellmann-implication of HS sum rule} in section 3 below.
 \end{rem}
 \begin{rem} Comparing the coefficient of $z^2$ in \eqref{Feynman-Hellmann-sum-rule-quadratic} we have the identity or choosing $f=1$ in \eqref{Feynman-Hellmann-sum-rule-general}we get
 	\begin{equation}\label{Feynman-Hellmann-sum-rule-constant}
 		\sum_{\lambda_j\in J} \ddot{\lambda}_j  = \sum_{\lambda_j\in J} \langle \ddot{H}u_j,u_j\rangle
 		+2\sum_{\lambda_j\in J}\sum_{\lambda_k\notin J}(\lambda_j-\lambda_k)^{-1}\big|\langle \dot{H} u_j,u_k\rangle\big|^2
 	\end{equation}
 which corresponds to the expression derived in second order perturbation theory (see e.g. \cite{Sch1955}, eq. (25.12), p.153 or eq.(25.14) p.154).
 \end{rem}
 
 From \eqref{Feynman-Hellmann-sum-rule-constant} we will prove the following result for the bottom of the spectrum $\sigma(H(\tau))$ of a family of semi-bounded operators $H(\tau)$:
 
  \begin{thm}\label{thm-2nd-order-Feynman-Hellmann-squeezing}
  Let $d=d(\tau):=\inf J^c-\sup J>0$ and suppose that $\dot{H} u_j$, $\dot{H}^2 u_j$ are in the domain of $H$. Then
  \begin{equation}\label{Feynman-Hellmann-sum-rule-constant-squeezing}
  0\leq \sum_{\lambda_j\in J} \langle \ddot{H}u_j,u_j\rangle	- \sum_{\lambda_j\in J} \ddot{\lambda}_j  
  	\leq d_m^{-2}\sum_{\lambda_j\in J}\langle [\dot{H},[H,\dot{H}]] u_j,u_j\rangle
  \end{equation}
  where $[A,B]=AB-BA$ denotes the commutator of $A$ and $B$. In particular, when $\dot{H}$ commutes with $H$ we have
  \begin{equation}
  	\sum_{\lambda_j\in J} \langle \ddot{H}u_j,u_j\rangle - \sum_{\lambda_j\in J} \ddot{\lambda}_j =0
  \end{equation} 
  and consequently $\ddot{\lambda}_j=\langle \ddot{H}u_j,u_j\rangle$ for all $j$.
   \end{thm}
   
   \begin{rem} For $2\times 2$ matrices $H$, $d=\lambda_2-\lambda_1>0$ the second inequality in \eqref{Feynman-Hellmann-sum-rule-constant-squeezing} is always an equality.
   \end{rem}
 The paper is organized as follows: In section 2 we prove theorem \ref{thm-Feynman-Hellmann theorem in quadratic sum rules} on the monotonicity of quadratic Riesz-means and show applications to Lieb-Thirring inequalities for Schr\"{o}dinger operators. In section 3 we prove the main results about the second-order Feynman-Hellmann sumr rules and show further consequences within the abstract setting. In section 4 we apply the sum rules to hermitian matrices, in particular to the linear combination $H=(1-t)A+tB$,  discussing known results in our framework and showing new inequalities for eigenvalues and traces of functions of $H$. In section 5 we illustrate our abstract results studying the one-parameter family of eigenvalue problems given by
 \begin{equation}
 	H=-\frac{d^2}{dx^2}-\frac{1}{4x^2}+\frac{\nu^2}{x^2}
 \end{equation}
 on $[0,1]$ with Dirichlet boundary conditions and $\nu\geq 0$. Its eigenvalues are given by the squares of the zeros $j_{\nu,k}$ of the Bessel function $J_{\nu}(x)$. We present some new concavity properties and new sharp inequalities on the spacings of $j_{\nu,k}^2$ and derivatives of moments of $j_{\nu,k}$.
 
\section{The Feynman-Hellman theorem in sum rules}
We start with the proof of theorem \ref{thm-Feynman-Hellmann theorem in quadratic sum rules}:
\begin{proof}[Proof of theorem \ref{thm-Feynman-Hellmann theorem in quadratic sum rules}]
	We apply the quadratic sum rule \eqref{HS-sumrule-discrete} to each $G_{\alpha}$ and to the spectral subset
	$J=\{\lambda_j: \lambda_j\leq z\}$ for $z\in\mathbb{R}$. The right hand side of \eqref{HS-sumrule-discrete}  is $\leq 0$.
	Summing over $\alpha$ and using the conditions eqs. \eqref{eq-energy-bounds-on-commutators-1},  \eqref{eq-energy-bounds-on-commutators-2} we get the inequality
	\begin{equation}
		\sum_j(z-\lambda_j(\tau))_{+}^2-(z-\lambda_j(\tau))_{+}(\eta(\tau)\lambda_j+ \theta(\tau)\dot{\lambda}_j)\leq 0
	\end{equation}
	for all real $z$. On the other hand, we have
	\begin{equation*}
		\begin{split}
		&\frac{d}{d\tau}\sum_jA^{-2}(\tau)(zB(\tau)-\lambda_j(\tau))_{+}^2\\
		&=-2\sum_j\dot{A}(\tau)A^{-3}(\tau)(zB(\tau)-\lambda_j(\tau))_{+}^2-A^{-2}(\tau)(zB(\tau)-\lambda_j(\tau))_{+}(z\dot{B}(\tau)-\dot{\lambda}_j(\tau).\\
		\end{split}
	\end{equation*}
	inserting the differential equations for $A$ and $B$ we get (suppressing the argument $\tau$)
	\begin{equation*}
		\begin{split}
			&\frac{d}{d\tau}\sum_jA^{-2}(\tau)(zB-\lambda_j)_{+}^2\\
			&=\frac{2}{A^2\theta}\sum_j(1+\eta)(zB-\lambda_j)_{+}^2-(zB-\lambda_j)_{+}(zB\eta+\theta \dot{\lambda}_j)\\
			&=\frac{2}{A^2\theta}\sum_j(zB-\lambda_j)_{+}^2-(zB-\lambda_j)_{+}(\eta \lambda_j+\theta \dot{\lambda}_j)\\
		\end{split}
	\end{equation*}
	proving the theorem.
\end{proof}
\begin{rem}
	According to \cite{HarStu2010} the theorem also holds in the presence of continuous spectrum and eigenvalues below the continuous spectrum.
\end{rem}
We also note the following corollary for the negative eigenvalues of $H$:
\begin{cor}
	Under the assumptions of theorem \ref{thm-Feynman-Hellmann theorem in quadratic sum rules} we have
	\begin{equation}
			\begin{split}
	\frac{d}{d\tau}\sum_jA^{-2}(-\lambda_j)_{+}^2\leq 0&\quad\text{if $\theta> 0$ }\\
	\frac{d}{d\tau}\sum_jA^{-2}(-\lambda_j)_{+}^2\geq 0&\quad\text{if $\theta< 0$ }
		\end{split}
	\end{equation}
\end{cor}
In the following we give some applications of these abstract results to Schr\"{o}dinger operators. For applications to the zeros of Bessel functions we refer to section 4 below.
\subsection{Lieb-Thirring inequalities for Schr\"{o}dinger operators}
We consider the eigenvalues $\lambda_j(\tau)$ of a one-parameter
family of Schr\"{o}dinger operators
\begin{equation}\label{Schrödinger operator H-tau on entire space}
	H(\tau)=-\tau\Delta+V(x)
\end{equation}
on $\mathbb{R}^d$ for constants $\tau>0$. As in \cite{Stu2010} we
suppose that $V(x)$ is a continuous function of compact support and we denote its negative part by $V_{-}(x)$. It is a well-known
fact (see e.g. \cite{BlSt1996}, \cite{FLW2022} and references therein) that
for all $\sigma\geq 0$
\begin{equation}\label{sc-limit}
	\underset{\tau\rightarrow 0+}{\lim}\tau^{\frac{d}{2}}\;
	\sum_{E_j(\tau)<0}(-E_j(\tau))^{\sigma}=L_{\sigma,d}^{cl}\int_{\mathbb{R}^d}V_{-}(x)^{\sigma+\frac{d}{2}}\;dx
\end{equation}
with $L_{\sigma,d}^{cl}$, called the classical constant, given by
\begin{equation}\label{sc-constant}
	L_{\sigma,d}^{cl}=(4\pi)^{-\frac{d}{2}}\frac{\Gamma(\sigma+1)}{\Gamma(\sigma+\frac{d}{2}+1)}.
\end{equation}
Lieb-Thirring inequalities are inequalities of the form
\begin{equation}\label{LT-ineq}
	\tau^{\frac{d}{2}}\;
	\sum_{E_j(\tau)<0}(-E_j(\tau))^{\sigma}\leq L_{\sigma,d}\int_{\mathbb{R}^d}V_{-}(x)^{\sigma+\frac{d}{2}}\;dx
\end{equation}
for some constant $L_{\sigma,d}\geq L_{\sigma,d}^{cl}$.
Let $G$ be the multiplication operator by the function $G=G(x)$. When $G(x)=e^{iqx}$ where $q\in \mathbb{R}$ with euclidean length $|q|\neq 0$, $G$ defines a bounded (even unitary) operator which maps the domain of $H$ into itself while in \cite{Stu2010} the operator $G$ was chosen to be the multiplication by a coordinate function which is obviously an unbounded operator.
With $G(x)=e^{iqx}$ we have the commutators
\begin{equation*}
	[G^*,[H,G]]=[G,[H,G^*]]=2\tau |q|^2
\end{equation*}
and
\begin{equation*}
\langle[H,G]\phi_j,[H,G]\phi_j\rangle =\langle[H,G^*]\phi_j,[H,G^*]\phi_j\rangle=\tau^2|q|^4+4\tau^2\langle q\nabla\phi_j,q\nabla\phi_j\rangle
\end{equation*}
for all eigenfunctions $\phi_j$. Therefore by the sum rule \eqref{HS-sumrule-discrete} (see also \cite{HarStu2011}, section 3, eq. (3.5)  )  we have
\begin{equation}\label{eq-sum-rule-Lieb-Thirring-JS-1}
\sum_{E_j(\tau)<0}(-E_j(\tau))^{2}-(-E_j(\tau))(\tau|q|^2+4\tau\langle q\nabla\phi_j,q\nabla\phi_j\rangle)\leq 0.
\end{equation}
Now let $e_{\alpha}$ the canonical orthonormal basis of $\mathbb{R}^d$. We choose
\begin{equation}
G_{\alpha}(x)=e^{i|q|e_{\alpha}x}.
\end{equation}
By eq. \eqref{eq-sum-rule-Lieb-Thirring-JS-1} after summing over $\alpha$ we obtain

\begin{equation}\label{eq-sum-rule-Lieb-Thirring-JS}
	\sum_{E_j(\tau)<0}(-E_j(\tau))^{2}-(-E_j(\tau))(\frac{\tau|q|^2}{d}+\frac{4\tau}{d}\langle \nabla\phi_j,\nabla\phi_j\rangle)\leq 0.
\end{equation}
Since by the Feynman-Hellman theorem $\dot{E}_j(\tau)=\langle \nabla\phi_j,\nabla\phi_j\rangle$ we apply theorem \ref{thm-Feynman-Hellmann theorem in quadratic sum rules} with $\eta=0$ and $\theta=-\,\frac{4\tau}{d}<0 $. Hence $A(\tau)=\tau^{-d/4}$ and therefore $\displaystyle \tau^{d/2}	\sum_{E_j(\tau)<0}(-E_j(\tau))^{2}$ is non-increasing in $\tau>0$ and by the semi-classical limit \eqref{sc-limit} we have the sharp Lieb-Thirring inequality \eqref{LT-ineq} for $\sigma=2$  and $L_{2,d}= L_{2,d}^{cl}$ and as shown in \cite{Stu2010} for all $\sigma\geq 2$.

\subsection{Lieb-Thirring inequalities for Schr\"{o}dinger operators with surface potentials on a hyperplane}
We consider the operator
\begin{equation}\label{H-lambda-hyperplane}
	H(\lambda v) u = -\Delta u\quad\text{in}\;\mathbb{R}^{d+1}_{+} := \{(x,
	y) : x \in \mathbb{R}^{d}, y > 0\}
\end{equation}
together with boundary conditions of the third type
\begin{equation}\label{H-bc-hyperplane}
	-\frac{\partial u}{\partial y}-\lambda vu=0 \quad
	\text{on}\;\mathbb{R}^{d}\times\{0\}
\end{equation}
where $v$ denotes a real-valued function on $\mathbb{R}^{d}$. As
indicated in \cite{FL2007} under the assumption that $v$ is
form-compact with respect to $\sqrt{-\Delta}$ in
$L^2(\mathbb{R}^{d})$, then $H(\tau v)$ can be defined as a
self-adjoint operator in $L^2(\mathbb{R}^{d+1}$ by means of the
quadratic form
\begin{equation}\label{H-form-hyperplane}
	\int_{\mathbb{R}^{d+1}_{+}}|\nabla_{x} u|^2+|\nabla_{y} u|^2\;dxdy-\tau\;\int_{\mathbb{R}^{d}}v(x)|u|^2(x,0)\;dx
\end{equation}
and in this case the negative spectrum of $H(\lambda v)$ consists
of eigenvalues $E_j(\tau)$ of finite multiplicities. We apply our theorem \ref{thm-Feynman-Hellmann theorem in quadratic sum rules} to prove the following Lieb-Thirring inequality:
\begin{thm} Let $v(x)$ be form-compact with respect to $\sqrt{-\Delta}$ in
	$L^2(\mathbb{R}^{d})$. Then for any $\sigma\geq 2$ the mapping
	\begin{equation}\label{diff-ineq-2}
		\lambda\mapsto\tau^{-4-d}\;
		\sum_{E_j(\tau)<0}(-E_j(\tau))^{2}
	\end{equation}
	is increasing for all $\tau>0$. Consequently
	\begin{equation}\label{LT-2-sharp}
		\sum_{E_j(\tau)<0}(-E_j(\tau))^{2}\leq
		L_{2,d}^{cl}\;\tau^{4+d}\;\int_{\mathbb{R}^d}v_{+}(x)^{4+d}\;dx
	\end{equation}
	for all $\tau>0$.
\end{thm}
\begin{proof} As before we apply the sum rule \eqref{HS-sumrule-discrete}  to $G_{\alpha}(x)=e^{i|q|e_{\alpha}x}$, $\alpha=1,\ldots,d$. After
	averaging over $\alpha$ we get
	\begin{equation}\label{Riesz-mean-ineq-1}
		\sum_{E_j(\tau)<0}(-E_j(\lambda))^{2}-\frac{4}{d}\sum_{E_j(\tau)<0}(-E_j(\tau)) \int_{\mathbb{R}^{d+1}_{+}}|\nabla_{x}
		\phi_j|^2\;dxdy\leq 0
	\end{equation}
	where $\phi_j$ denotes the eigenfunction corresponding to
	$E_j(\tau)$. By the Feynman-Hellmann theorem we have
	\begin{equation}\label{FH}
		\frac{d}{d\tau}\;E_j(\tau)=-\;\int_{\mathbb{R}^{d}}v(x)|\phi_j|^2(x,0)\;dx
	\end{equation}
	from which we get
	\begin{equation}\label{kinetic-energy-relation-1}
		\int_{\mathbb{R}^{d+1}_{+}}|\nabla_{x} \phi_j|^2+|\nabla_{y}
		\phi_j|^2\;dxdy=E_j(\tau)-\tau\;\frac{d}{d\tau}\;E_j(\tau).
	\end{equation}
	On the other hand, by the virial theorem
	\begin{equation}\label{virial-theorem}
		\int_{\mathbb{R}^{d+1}_{+}}|\nabla_{x} \phi_j|^2-|\nabla_{y}
		\phi_j|^2\;dxdy=E_j(\tau).
	\end{equation}
	Combining \eqref{kinetic-energy-relation-1} and
	\eqref{virial-theorem} we get
	\begin{equation}\label{kinetic-energy-relation-2}
		\int_{\mathbb{R}^{d+1}_{+}}|\nabla_{x}
		\phi_j|^2\;dxdy=E_j(\tau)-\frac{\tau}{2}\frac{d}{d\tau}\;E_j(\tau).
	\end{equation}
	Inserting into \eqref{Riesz-mean-ineq-1} we obtain after
	rearranging terms and simplification the differential inequality
	\begin{equation}\label{Riesz-mean-diff-ineq}
		\frac{d}{d\tau}\bigg(\frac{\sum_{E_j(\tau)<0}(-E_j(\tau))^{2}}{\tau^{4+d}}\bigg)\geq 0
	\end{equation}
	which proves the theorem.
\end{proof}
\subsection{Schr\"{o}dinger operators on closed compact hypersurfaces}
Let $M$ be a closed compact smooth $d$-dimensional manifold smoothly immersed in $\mathbb{R}^{d+1}$ with bounded mean curvature. On $M$ we consider the Laplace-Beltrami operator with a potential given by the square of the total mean curvature $h$:
\begin{equation}\label{H-beta}
	H_{\beta}=-\Delta+\frac{1-\beta}{4}\,h^2,\quad\beta\in\mathbb{R}.
\end{equation}
The study of this problem was introduced in \cite{Har2007}, see also \cite{HarStu2011} for the notations. We denote its eigenvalues by $\lambda_k=\lambda_k(\beta)$ in increasing order where $k$ is a positive integer and corresponding normalized eigenfunctions $u_k$. When $\beta=1$ then $\lambda_1=0$ and $u_1$ is constant. Denoting the scalar product in $L^2(M)$ by $\langle\cdot,\cdot\rangle$ we have the variational identity
\begin{equation}\label{eigenvalue-id-mean-curvature}
	\lambda_k=\langle u_k,(-\Delta+\frac{1-\beta}{4}\,h^2)u_k\rangle
\end{equation}
and by the Feynman-Hellmann theorem
\begin{equation}\label{Feynman-Hellmann-relation-mean-curvature-1}
	\dot{\lambda}_k=\frac{d\,\lambda_k}{d\beta}=-\,\frac{1}{4}\langle u_k,h^2u_k\rangle.
\end{equation}
In \cite{Har2007}, \cite{HarStu2011} the following inequality was derived from the quadratic sum rule \eqref{HS-sumrule-discrete} with $G$ chosen to the multiplication by a coordinate function. For all $z,\beta\in\mathbb{R}$:
\begin{equation}\label{HS-sum-rule-ineq-mean-curvature-1}
	\sum_{j}(z-\lambda_j)_{+}^2\leq \frac{4}{d} \sum_{j}(z-\lambda_j)_{+}\langle u_j,(-\Delta+\frac{1}{4}\,h^2)u_j\rangle.
\end{equation}
The main interest in \cite{Har2007}, \cite{HarStu2011} was the case $\beta=1$, that is eigenvalues of the Laplacian and the mean curvature was replaced by its upper bound in the sum rule \eqref{HS-sum-rule-ineq-mean-curvature-1}. Here we exploit the the variational identity  \eqref{eigenvalue-id-mean-curvature} and Feymman-Hellmann relation \eqref{Feynman-Hellmann-relation-mean-curvature-1} and get
\begin{equation}\label{eq-eigenvalue-mean-curvature-identity}
	\langle u_j,(-\Delta+\frac{1}{4}\,h^2)u_j\rangle= \lambda_j-\beta\dot{\lambda}_j.
\end{equation}
Inequality \eqref{HS-sum-rule-ineq-mean-curvature-1} is therefore equivalent to
\begin{equation}\label{HS-sum-rule-ineq-mean-curvature-2}
	\sum_{j}(z-\lambda_j)_{+}^2\leq \frac{4}{d} \sum_{j}(z-\lambda_j)_{+}(\lambda_j-\beta\dot{\lambda}_j)
\end{equation}
and after reorganizing terms as in \cite{HarStu2011} we get
\begin{equation}\label{HS-sum-rule-ineq-mean-curvature-3}
	\frac{d+4}{2}\sum_{j}(z-\lambda_j)_{+}^2\leq 2z\sum_{j}(z-\lambda_j)_{+}- 2\beta\sum_{j}(z-\lambda_j)_{+}\dot{\lambda}_j.
\end{equation}
We rewrite \eqref{HS-sum-rule-ineq-mean-curvature-3} in different ways. First we note that for $z\geq 0$ inequality \eqref{HS-sum-rule-ineq-mean-curvature-3} is equivalent to the linear partial differential inequality
\begin{equation}\label{HS-sum-rule-ineq-mean-curvature-pd-ineq-1}
	z\frac{\partial}{\partial z}\frac{\sum_{j}(z-\lambda_j)_{+}^2}{z^{2+d/2}}+ \beta\,\frac{\partial}{\partial\beta}\frac{\sum_{j}(z-\lambda_j)_{+}^2}{z^{2+d/2}}\geq 0\,.
\end{equation}
Replacing $z$ by $\beta z$ in \eqref{HS-sum-rule-ineq-mean-curvature-3} we get
\begin{equation}\label{HS-sum-rule-ineq-mean-curvature-4}
	\frac{d+4}{2}\sum_{j}(\beta z-\lambda_j)_{+}^2\leq 2\beta z\sum_{j}(\beta z-\lambda_j)_{+}- 2\beta\sum_{j}(\beta z-\lambda_j)_{+}\dot{\lambda}_j
\end{equation}
that is for any $z$ fixed
\begin{equation*}
	\frac{d+4}{2}\sum_{j}(\beta z-\lambda_j)_{+}^2\leq \beta\frac{d}{d\beta}\sum_{j}(\beta z-\lambda_j)_{+}^2
\end{equation*}
or equivalently the linear differential inequality for $\beta>0$:
\begin{equation}\label{HS-sum-rule-ineq-mean-curvature-od-ineq-1}
	\frac{d}{d\beta}\frac{\sum_{j}(\beta z-\lambda_j)_{+}^2}{\beta^{2+d/2}}\geq 0.
\end{equation}
Consequently we have for all $\beta_2\geq \beta_1>0$ and all $z> 0$ the following inequalities:
\begin{equation}\label{HS-sum-rule-ineq-mean-curvature-ineq-1}
	\frac{\sum_{j}\big(\beta_2 z-\lambda_j(\beta_2)\big)_{+}^2}{(\beta_2z)^{2+d/2}}\geq \frac{\sum_{j}\big(\beta_1 z-\lambda_j(\beta_1)\big)_{+}^2}{(\beta_1z)^{2+d/2}}
\end{equation}
or equivalently
\begin{equation}\label{HS-sum-rule-ineq-mean-curvature-final-ineq-2}
	\frac{\sum_{j}\big(z-\lambda_j(\beta_2)\big)_{+}^2}{(\beta_2z)^{2+d/2}}\geq \frac{\sum_{j}\big(\frac{\beta_1 z}{\beta_2}-\lambda_j(\beta_1)\big)_{+}^2}{(\beta_1z)^{2+d/2}}.
\end{equation}
For $z\geq 0$  we consider the operator
\begin{equation*}
	H_{\beta,z}=H_{\beta}-\beta z=-\Delta+\frac{1}{4}\,h^2-\beta(\frac{1}{4}\,h^2+z).
\end{equation*}
Its negative eigenvalues correspond to the eigenvalues of $H_{\beta}$ below $\beta z$. Let $N_{-}(H_{\beta,z})$ denote the number of negative eigenvalues of $H_{\beta,z}$. It satisfies the Weyl-limit
\begin{equation*}
	\underset{\beta\rightarrow \infty}{\lim}\,\frac{N_{-}(H_{\beta,z})}{\beta^{d/2}}=(2\pi)^{-d}B_d\int_M(\frac{1}{4}\,h^2+z)_{+}^{d/2}\,dV
\end{equation*}
from which we easily get the Weyl-limits for the the sums of the squares of the negative eigenvalues of $H_{\beta,z}$. Therefore letting $\beta_2$ tend to infinity in \eqref{HS-sum-rule-ineq-mean-curvature-ineq-1} we get the Weyl estimate
\begin{equation}\label{HS-sum-rule-ineq-mean-curvature-Weyl-estimate-1}
	\sum_{j}\big(\beta_1 z-\lambda_j(\beta_1)\big)_{+}^2\leq (2\pi)^{-d}B_d\,\frac{8\,\beta_1^{2+d/2}}{(d+2)(d+4)}\int_M(\frac{1}{4}\,h^2+z)_{+}^{2+d/2}\,dV
\end{equation}
or equivalently
\begin{equation}\label{HS-sum-rule-ineq-mean-curvature-Weyl-estimate-2}
	\sum_{j}\big(z-\lambda_j(\beta)\big)_{+}^2\leq (2\pi)^{-d}B_d\,\frac{8}{(d+2)(d+4)}\int_M(\frac{\beta}{4}\,h^2+z)_{+}^{2+d/2}\,dV.
\end{equation}
In particular, for $\beta=1$ that is for the eigenvalues of the Laplace-Beltrami $\lambda_j(1)$ we have

\begin{equation}\label{HS-sum-rule-ineq-mean-curvature-estimate-3}
	\sum_{j}\big(z-\lambda_j(1)\big)_{+}^2\leq (2\pi)^{-d}B_d\,\frac{8}{(d+2)(d+4)}\int_M(\frac{1}{4}\,h^2+z)_{+}^{2+d/2}\,dV
	\end{equation}
which improves the estimate in \cite{HarStu2011}.

\section{A sum rule for second derivatives of eigenvalues}
We prove our main results Theorem \ref{thm-second order Feynman-Hellmann theorem}, Theorem \ref{thm-2nd-order-Feynman-Hellmann-squeezing} and present further consequences. 
\subsection{Proof of Theorem \ref{thm-second order Feynman-Hellmann theorem}}
We recall that deriving the Feyman-Hellmann relation $\dot{\lambda}_{j}=\langle u_{j},\dot{H} u_{j}\rangle$ with respect to $\tau$ we have the second derivative formula
\begin{equation}\label{Feynman-Hellmann theorem-2nd-derivative-short-bis}
	\ddot{\lambda}_{j}=\langle u_{j},\ddot{H} u_{j}\rangle+2\langle \dot{u}_{j},\dot{H} u_{j}\rangle.
\end{equation}
We decompose the last term in \eqref{Feynman-Hellmann theorem-2nd-derivative-short-bis} applying the completeness of the eigenfunctions $u_j$ and the following gap formula for $\dot{H}$.

\begin{lem}
	Under the conditions of Theorem \ref{thm-second order Feynman-Hellmann theorem} for any pair of eigenfunctions $u_j,u_k$ the following gap formula holds:
	\begin{equation}\label{Feynman-Hellmann-gap-formula}
		\langle \dot{H}u_j,u_k\rangle=(\lambda_j-\lambda_k)\langle \dot{u}_j,u_k\rangle+\dot{\lambda}_j\delta_{jk}
	\end{equation}
	where $ \dot{u}_j$ denotes the derivative of $u_j$ with respect to the parameter $\tau$ and $\delta_{jk}$ is the Kronecker-Delta.
\end{lem}
\begin{proof}
	First we note that the normalization $\langle u_j,u_k\rangle=\delta_{jk}$ implies
	\begin{equation}\label{derivate-normalized-u_j}
		\langle \dot{u}_j,u_k\rangle+\langle u_j, \dot{u}_k\rangle=0.
	\end{equation}
	and therefore
	\begin{equation}
		\dot{\lambda}_j\delta_{jk}=\frac{d}{d\,\tau}\langle H u_j,u_k\rangle=\langle \dot{H}u_j,u_k\rangle + \langle H\dot{u}_j,u_k\rangle+\langle H u_j, \dot{u}_k\rangle
	\end{equation}
	from which we conclude by the self-adjointness of $H$ and eq.\eqref{derivate-normalized-u_j}.
\end{proof}
 \begin{rem}
 	There is also integral version of the gap formula for $\dot{H}$ corresponding to the integral version of the Feynman-Hellmann theorem given in eq. \eqref{Feynman-Hellman-integral-version}
 	\begin{equation}\label{Feynman-Hellman-integral-version-gap-formula}
 		\begin{split}
 		\langle  u_{j,\tau+\delta},(H_{\tau+\delta}-H_{\tau})u_{k,\tau}\rangle&=(\lambda_{j,\tau}-\lambda_{k,\tau})\langle u_{j,\tau+\delta}-u_{j,\tau},u_{k,\tau}\rangle\\
 		&\quad +(\lambda_{j,\tau+\delta}-\lambda_{j,\tau})\langle u_{j,\tau+\delta},u_{k,\tau}\rangle.\\
 		\end{split}
 	\end{equation}
 \end{rem}
By the second derivative formula eq.\eqref{Feynman-Hellmann theorem-2nd-derivative-short-bis} the gap formula will imply the following
\begin{prop}Under the above conditions for each positive integer $j$:
	\begin{equation}\label{second-derivative-sum-rule-1}
		\langle \ddot{H}u_j,u_j\rangle-\ddot{\lambda}_j =2\sum_{\lambda_k\neq\lambda_j}(\lambda_k-\lambda_j)\big|\langle \dot{u}_j,u_k\rangle\big|^2=2\sum_{\lambda_k\neq\lambda_j}\frac{\big|\langle \dot{H}u_j,u_k\rangle\big|^2}{\lambda_k-\lambda_j}.
	\end{equation}
\end{prop}
\begin{proof} Using the completeness of the eigenfunctions and the relation $\langle \dot{u}_j,u_j\rangle=0$ we get by the gap formula \eqref{Feynman-Hellmann-gap-formula}:
	\begin{equation*}
		\begin{split}
			\langle \dot{H}\dot{u}_j,u_j\rangle&=\langle \dot{u}_j,\dot{H}u_j\rangle\\
			&=\sum_{\lambda_k}\langle \dot{u}_j,u_k\rangle\langle u_k,\dot{H}u_j\rangle\\
			&=\sum_{\lambda_k\neq\lambda_j}\langle \dot{u}_j,u_k\rangle\langle u_k,\dot{H}u_j\rangle\\
			&=-\sum_{\lambda_k\neq\lambda_j}(\lambda_k-\lambda_j)\big|\langle \dot{u}_j,u_k\rangle\big|^2\\
			&=-\sum_{\lambda_k\neq\lambda_j}\frac{\big|\langle \dot{H}u_j,u_k\rangle\big|^2}{\lambda_k-\lambda_j}
		\end{split}
	\end{equation*}
	and the proposition follows from the second derivative formula eq.\eqref{Feynman-Hellmann theorem-2nd-derivative-short-bis} 
\end{proof}

We are now in position to prove Theorem \ref{thm-second order Feynman-Hellmann theorem}.
\begin{proof}[Proof of Theorem \ref{thm-second order Feynman-Hellmann theorem}]
By the sum rule for the second derivative eq.\eqref{second-derivative-sum-rule-1} we have
\begin{equation*}
	\begin{split}
		\sum_{\lambda_j} f(\lambda_j)(\ddot{\lambda}_j -\langle \ddot{H}u_j,u_j\rangle)&=-\underset{\lambda_j\neq\lambda_k}{\sum\sum}\frac{2\,f(\lambda_j)}{\lambda_k-\lambda_j}\big|\langle u_k,\dot{H}u_j\rangle\big|^2\\
		&=\underset{\lambda_j\neq\lambda_k}{\sum\sum}\frac{f(\lambda_k)-f(\lambda_j)}{\lambda_k-\lambda_j}\big|\langle u_k,\dot{H}u_j\rangle\big|^2\\
	\end{split}
\end{equation*}
where in the last line we have symmetrized the double sum. When we sum only $\lambda_j\in J$ for a subset $J$ of the spectrum we have
\begin{equation}\label{Feynman-Hellmann-sum-rule-general-splittig-with-J}
	\begin{split}
		\sum_{\lambda_j\in J} f(\lambda_j)(\ddot{\lambda}_j -\langle \ddot{H}u_j,u_j\rangle)&=\underset{\lambda_j\neq\lambda_k}{\sum_{\lambda_j\in J}\sum_{\lambda_k\in J}}\frac{f(\lambda_k)-f(\lambda_j)}{\lambda_k-\lambda_j}\big|\langle u_k,\dot{H}u_j\rangle\big|^2\\
		&-\sum_{\lambda_j\in J}\sum_{\lambda_k\notin J}\frac{2\,f(\lambda_j)}{\lambda_k-\lambda_j}\big|\langle u_k,\dot{H}u_j\rangle\big|^2.\\
	\end{split}
\end{equation}
When $f'$ is concave we use the inequality
\begin{equation*}
	\frac{f(\lambda_k)-f(\lambda_j)}{\lambda_k-\lambda_j}\geq \frac{f'(\lambda_j)+f'(\lambda_k)}{2}
\end{equation*}
which yields using the symmetry in the first double sum
\begin{equation*}
	\begin{split}
		&\sum_{\lambda_j\in J} f(\lambda_j)(\ddot{\lambda}_j -\langle \ddot{H}u_j,u_j\rangle)\\
		&\geq\underset{\lambda_j\neq\lambda_k}{\sum_{\lambda_j\in J}\sum_{\lambda_k\in J}}f'(\lambda_j)\big|\langle u_k,\dot{H}u_j\rangle\big|^2
		-\sum_{\lambda_j\in J}\sum_{\lambda_k\notin J}\frac{2\,f(\lambda_j)}{\lambda_k-\lambda_j}\big|\langle u_k,\dot{H}u_j\rangle\big|^2\\
		&=\sum_{\lambda_j\in J}\sum_{\lambda_k}f'(\lambda_j)\big|\langle \dot{H}u_j,u_k\rangle-\dot{\lambda}_j\delta_{jk}\big|^2\\
		&-\sum_{\lambda_j\in J}\sum_{\lambda_k\notin J}\frac{2\,f(\lambda_j)+f'(\lambda_j)(\lambda_k-\lambda_j)}{\lambda_k-\lambda_j}\big|\langle u_k,\dot{H}u_j\rangle\big|^2\\
		&=\sum_{\lambda_j\in J}f'(\lambda_j)\big(\langle \dot{H}u_j,\dot{H}u_j\rangle-\dot{\lambda}_j^2\big)\\
		&-\sum_{\lambda_j\in J}\sum_{\lambda_k\notin J}\frac{2\,f(\lambda_j)+f'(\lambda_j)(\lambda_k-\lambda_j)}{\lambda_k-\lambda_j}\big|\langle u_k,\dot{H}u_j\rangle\big|^2\\
	\end{split}
\end{equation*}
which is inequality \eqref{Feynman-Hellmann-sum-rule-general-ineq-f-prime-concave}. The above inequality becomes an equality when $f'$ is linear, i.e. $f$ is of the form $f(\lambda_j)=a\lambda_j^2+b\lambda_j+c$, and choosing $f(\lambda_j)=(z-\lambda_j)^2$ we get after reorganizing terms the quadratic second-order Feynman-Hellmann sum rule eq. \eqref{Feynman-Hellmann-sum-rule-quadratic}.
\end{proof}

\subsection{Proof of Theorem \ref{thm-2nd-order-Feynman-Hellmann-squeezing}}
 Let $d:=\inf J^c-\sup J>0$. We consider the sum rule \eqref{Feynman-Hellmann-sum-rule-constant} which corresponds to the coefficient of $z^2$ in the quadratic second-order Feynman-Hellmann sum rule eq. \eqref{Feynman-Hellmann-sum-rule-quadratic}, i.e.
 \begin{equation}
 	\sum_{\lambda_j\in J} \langle \ddot{H}u_j,u_j\rangle -\ddot{\lambda}_j  = 
 	2\sum_{\lambda_j\in J}\sum_{\lambda_k\notin J}(\lambda_k-\lambda_j)^{-1}\big|\langle \dot{H} u_j,u_k\rangle\big|^2.
 \end{equation}
Obviously,
\begin{equation*}
	2\sum_{\lambda_j\in J}\sum_{\lambda_k\notin J}(\lambda_k-\lambda_j)^{-1}\big|\langle \dot{H} u_j,u_k\rangle\big|^2\leq d^{-2} \sum_{\lambda_j\in J}\sum_{\lambda_k\notin J}(\lambda_k-\lambda_j)^{+1}\big|\langle \dot{H} u_j,u_k\rangle\big|^2
\end{equation*}
and we evaluate the double sum by adding the abstract Thomas-Reiche-Kuhn sum rule \eqref{TRK-abstract} over $\lambda_{j}\in J$ with $G=\dot{H}$ which corresponds to the coefficient of $z^2$ in the quadratic sum rule \eqref{HS-sumrule-discrete} by Harrell and the author.

\subsection{Further results}
First of all, we show that the quadratic second order Feynman-Hellmann relation \eqref{Feynman-Hellmann-sum-rule-quadratic} implies the quadratic sum rule \eqref{HS-sumrule-discrete} for self-adjoint $G$ and more generally \eqref{Feynman-Hellmann-sum-rule-general} implies the abstract sum rule stated in \cite{HarStu2010}.

 \begin{thm}\label{thm-2nd-order-Feynman-Hellmann-implication of HS sum rule}
Let $H_0$ be a self-adjoint operator with domain $\mathcal{D}$ and spectrum consisting of eigenvalues $\lambda_{j}$ with corresponding eigenfunctions $\phi_j$. Let $G$ be another self-adjoint operator with domain $\mathcal{D}_G$ satisfying the additional domain hypothesis $G(\mathcal{D})\subseteq \mathcal{D}\subseteq \mathcal{D}_G$. Then for any subset $J$ of the spectrum the quadratic sum rule
\begin{equation}\label{HS-sumrule-discrete-from-Feynman-Hellmann}
	\begin{split}
		&-\;\frac1{2}\sum_{\lambda_j\in J}  (z-\lambda_j)^2\,\langle[G,[H_0,G]]\phi_j,\phi_j\rangle -\sum_{\lambda_j\in
			J}(z-\lambda_j)\,\langle[H_0,G]\phi_j,[H_0,G]\phi_j\rangle)\\
		&= \sum_{\lambda_j\in J}\sum_{\lambda_k\notin J}
		(z-\lambda_j)(z-\lambda_k)(\lambda_k-\lambda_j)\,|\langle G\phi_j,\phi_k\rangle|^2\\
	\end{split}
\end{equation}
holds.
\end{thm}
\begin{proof}
By Stone's theorem the unitary group $\displaystyle e^{iG\tau}$ is strongly continuous. We define the operator $H=H(\tau)$ by
\begin{equation}
	H=e^{-iG\tau}H_0e^{iG\tau}.
\end{equation}
Obvioulsy the eigenvalues of $H$ do not depend on $\tau$ and are given by $\lambda_{j}$ with corresponding eigenfunctions $u_j= e^{-iG\tau}\phi_j $ (which are real analytic in $\tau$ by the domain hypothesis on $G$). We compute
\begin{equation}
	\dot{H}=ie^{-iG\tau}[H_0,G]e^{iG\tau}, \quad 	\ddot{H}=e^{-iG\tau}[G,[H_0,G]]e^{iG\tau}.
\end{equation}
Inserting these facts into the quadratic second order Feynman-Hellmann sum rule eq.\eqref{Feynman-Hellmann-sum-rule-quadratic} we conclude.
\end{proof}

\begin{rem}
	The $u_j(\tau)$ satisfy the abstract Schr\"{o}dinger equation
	\begin{equation*}
	i\frac{d \, u}{d\tau}=Gu.
	\end{equation*}
	Therefore the second commutator $[G,[H_0,G]]$ is the expression in Ehrenfest's theorem for the expectation values of $H_0$ for solutions of the above Schr\"{o}dinger equation since (at least formally)
	\begin{equation*}
		\frac{d^2}{d\tau^2} \langle H_0u,u\rangle=\langle[G,[H_0,G]]u,u\rangle
	\end{equation*}
	for solutions of the above Schr\"{o}dinger equation. Note however, that in the applications of Ehrenfest's theorem \cite{Ehr1927} the role of $G$ and $H_0$ are interchanged in the sense that $G$ is a Schr\"{o}dinger operator of the form $-\Delta +V(x)$ and $H_0=x$ is the position operator so that $[G,[H_0,G]]=-2\nabla V(x)$ while in the sum rule applications  $H_0$ is Schr\"{o}dinger operator and $G$ typically a multiplication operator. Inserting superpositions of standing waves into Ehrensfest's theorem for the Schr\"{o}dinger operator $-\Delta +V(x)$ lead to sum rules different from those presented here with various applications in quantum mechanics: for example in \cite{GroStu1995} it is shown that the ground state of the (non-relativistic) hydrogen atom in a constant magnetic field has zero angular momentum, to our knowledge the only "pure" quantum mechanical proof of this fact known so far.
\end{rem}

\section{Inequalities for Eigenvalues and Traces of Hermitian Matrices}

\subsection{Inequalities for traces of Hermitian matrices.}
We apply the second order Feynman-Hellmann sum rules of Theorem \ref{thm-second order Feynman-Hellmann theorem} when $J$ represents the full spectrum of $H$. We provide not only alternative proofs of known trace inequalities but also correction terms sharpening some of these inequalities. In the following we note $Tr(H)$ the trace of $H$.
\begin{thm}\label{Theorem-trace-class-inequalities}
	Let $H=H_{\tau}$, $\tau$ real,  be an analytic one parameter family of $n\times n$ hermitian matrices. Let $F$ be a twice differentiable function defined on an interval containing the spectra of $H_{\tau}$ and let $f=F'$ be the derivative function of $F$.
	\begin{enumerate}
		\item Suppose that $F$ is convex (concave). Then
		\begin{equation}\label{eq-trace-convexity-concavity}
			\frac{d^2}{d\tau^2} Tr(F(H))\geq (\leq) Tr(f(H)\ddot{H})+\sum_{j=1}^nf'(\lambda_j)\dot{\lambda_j}^2
		\end{equation}
		with equality if $F$ is an affine linear function. In particular, if $\ddot{H}=0$ then $ Tr(F(H))$ is convex (concave) in $\tau$.
		\item Suppose that $f'=F''$ is convex (concave). Then
		\begin{equation}\label{eq-trace-convexity-concavity-operator-version}
			\frac{d^2}{d\tau^2} Tr(F(H))\leq (\geq) Tr(f(H)\ddot{H})+Tr(\dot{H}f'(H)\dot{H}).
		\end{equation}
		Equality holds iff $F$ is a polynomial of degree $3$.
		\item Suppose that $F$ is of class $C^4$ and that $F''''=f'''$ is convex (concave). Then
		\begin{equation}\label{eq-trace-improved-convexity-concavity-operator-version}
			\frac{d^2}{d\tau^2} Tr(F(H))\geq (\leq) Tr(f(H)\ddot{H})+Tr(\dot{H}f'(H)\dot{H})+\frac{1}{12}Tr(f'''(H)[H,\dot{H}]^2).
		\end{equation}
		Equality holds iff $F$ is a polynomial of degree $5$.
	\end{enumerate}
	
\end{thm}
\begin{rem}
	For certain classes of functions $F$ the inequalities of theorem \ref{Theorem-trace-class-inequalities} provide two-sided bounds for $\displaystyle \frac{d^2}{d\tau^2} Tr(F(H))$ which are consistent since for any positive function $g(\lambda)$ we have
	\begin{equation*}
	Tr(\dot{H}g(H)\dot{H})=\sum_{j=1}^{n}g(\lambda_j)\dot{\lambda}_j^2+g(\lambda_j)\langle(\dot{H}-\dot{\lambda}_j)u_j,(\dot{H}-\dot{\lambda}_j)u_j\rangle
	\end{equation*}
	by the Feynman-Hellmann theorem (for the first derivative) and the last term is nonnegative and in inequality \eqref{eq-trace-improved-convexity-concavity-operator-version} the operator $[H,\dot{H}]^2$ is negative.
\end{rem}
\begin{proof}
Since
\begin{equation*}
	\frac{d^2}{d\tau^2} Tr(F(H))=\sum_{j=1}^nf(\lambda_j)(\ddot{\lambda_j}-\langle \ddot{H}u_j ,u_j\rangle)+\sum_{j=1}^nf(\lambda_j)(\langle \ddot{H}u_j ,u_j\rangle) +f'(\lambda_j)\dot{\lambda_j}^2
\end{equation*}
we obtain replacing the first sum by the sum rule \eqref{Feynman-Hellmann-sum-rule-general} of Theorem \ref{thm-second order Feynman-Hellmann theorem} the following identity:
	\begin{equation}\label{Feynman-Hellmann-sum-rule-traces-general-1}
		\begin{split}
			\frac{d^2}{d\tau^2} Tr(F(H))&=\underset{\lambda_j\neq\lambda_k}{\sum\sum}\frac{f(\lambda_k)-f(\lambda_j)}{\lambda_k-\lambda_j}\big|\langle u_k,\dot{H}u_j\rangle\big|^2\\
			&\quad+\sum_{j=1}^nf(\lambda_j)(\langle \ddot{H}u_j ,u_j\rangle) +f'(\lambda_j)\dot{\lambda_j}^2.\\
		\end{split}
\end{equation}
Now, if $F$ is convex $f=F'$ is increasing and therefore the quotient in the first sum is always nonnegative proving the first statement of the theorem. If $f'$ is convex or concave we write
\begin{equation*}
\frac{f(\lambda_k)-f(\lambda_j)}{\lambda_k-\lambda_j}=\int_0^1f'((1-t)\lambda_j+t\lambda_k)\,dt
\end{equation*}
and if $F$ is of class $C^4$ we have
\begin{equation*}
	\frac{f(\lambda_k)-f(\lambda_j)}{\lambda_k-\lambda_j}=	\frac{f'(\lambda_k)+f'(\lambda_j)}{2}-\frac{(\lambda_k-\lambda_j)^2}{2}\int_0^1t(1-t)f'''((1-t)\lambda_j+t\lambda_k)\,dt.
\end{equation*}
We also note that
\begin{equation*}
	\begin{split}
	&\underset{\lambda_j\neq\lambda_k}{\sum\sum}\frac{f'(\lambda_k)+f'(\lambda_j)}{2}\big|\langle u_k,\dot{H}u_j\rangle\big|^2\\
	&=\underset{\lambda_j,\lambda_k}{\sum\sum}\frac{f'(\lambda_k)+f'(\lambda_j)}{2}\big|\langle u_k,\dot{H}u_j\rangle\big|^2-f'(\lambda_k)\delta_{jk}\big|\langle u_k,\dot{H}u_k\rangle\big|^2\\
	&=\sum_{j}^{n}f'(\lambda_j)\langle \dot{H}u_j,\dot{H}u_j\rangle-f'(\lambda_j)\dot{\lambda_j}^2.
	\end{split}
\end{equation*}
This proves the second statement of the Theorem. For the third statement we apply the gap formula
\begin{equation*}
(\lambda_k-\lambda_j)^2\big|\langle u_k,\dot{H}u_j\rangle\big|^2=\big|\langle u_k,[H,\dot{H}]u_j\rangle\big|^2
\end{equation*}
and the convexity/concavity of $f'''$.
\end{proof}
\subsection{Examples.} We consider the case $H=(1-\tau)A+\tau B$, i.e. $\ddot{H}=0$. First, we derive a consequence of the first inequality \eqref{eq-trace-convexity-concavity} of theorem \ref{Theorem-trace-class-inequalities}.
\begin{cor}\label{cor-trace-convexity-inequalities-1}
Under the conditions of theorem \ref{Theorem-trace-class-inequalities} suppose that $\ddot{H}=0$.Let either $F''>0$ or $F''<0$ and define the function $G$ by $G(\lambda):=\frac{F'(\lambda)^2}{F''(\lambda)}$. Then
\begin{equation}\label{eq-trace-convexity-inequalities-1}
	Tr(G(H))\frac{d^2}{d\tau^2} Tr(F(H))\geq \big(\frac{d}{d\tau}Tr(F'(H))\big)^2.
\end{equation}
In particular, we have the following properties for $H=(1-\tau)A+\tau B$ with $A,B$ positive matrices and $\tau \in [0,1]$.
\begin{enumerate}
	\item If $F>0$ is logarithmically convex then $Tr(F(H))$ is logarithmically convex, i.e.
	\begin{equation}
		Tr(F((1-\tau)A+\tau B))\leq 	Tr(F(A))^{1-\tau}	Tr(F(B))^{\tau}
	\end{equation}
	\item If $F>0$ and $F^{1/p}$ is convex for some $p\geq 1$ then $Tr(F(H))^{1/p}$ is convex, i.e.
	\begin{equation}
		Tr(F((1-\tau)A+\tau B)^{1/p})\leq 	(1-\tau)Tr(F(A))^{1/p}+\tau	Tr(F(B))^{1/p}
	\end{equation}
	\item If $F,F''>0$ and $F^{1/p}$ is concave for some $0<p\leq 1$ then $Tr(F(H))^{1/p}$ is concave, i.e.
	\begin{equation}
		Tr(F((1-\tau)A+\tau B)^{1/p})\geq 	(1-\tau)Tr(F(A))^{1/p}+\tau	Tr(F(B))^{1/p}
	\end{equation}
	\item If $F''>0$ and $F^{-1/p}$ is concave for some $p> 0$ then $Tr(F(H))^{-1/p}$ is concave, i.e.
	\begin{equation}
		Tr(F((1-\tau)A+\tau B)^{-1/p})\geq 	(1-\tau)Tr(F(A))^{-1/p}+\tau	Tr(F(B))^{-1/p}
	\end{equation}
	\item If $F''>0$ and $G(\lambda)\leq a$  for some $a>0$ then $\exp(-\frac{Tr(F(H))}{an}$ is concave. In particular, for $F(\lambda)=\-ln\lambda$ we get that $\det ((1-\tau)A+\tau B))^{1/n}$ is concave in $\tau$, i.e.
	\begin{equation}\label{concavity-of-determinants}
	\det ((1-\tau)A+\tau B))^{1/n}\geq (1-\tau)\det (A)^{1/n}+\tau\det (B)^{1/n}.
	\end{equation}
\end{enumerate}
\end{cor}
\begin{proof}
	Inequality \eqref{eq-trace-convexity-inequalities-1} follows from applying the Cauchy-Schwarz inequality
	\begin{equation*}
		\sum_{j=1}^nF''(\lambda_j)\dot{\lambda_j}^2\geq \bigg(\sum_{j=1}^nF'(\lambda_j)\dot{\lambda_j}\bigg)^2/\sum_{j=1}^n\frac{F'(\lambda_j)^2}{F''(\lambda_j)}
	\end{equation*}
	in the case $F''>0$ and similarly when $F''<0$. The proof of other inequalities is elementary rewriting the differential inequality \eqref{eq-trace-convexity-inequalities-1} for $Tr(F(H))$ in an appropriate way. For example, if $F$ is log-convex, then $F''F\geq F'^2$. Hence by inequality \eqref{eq-trace-convexity-inequalities-1} we have
	\begin{equation*}
		Tr(F(H))\frac{d^2}{d\tau^2} Tr(F(H))\geq \big(\frac{d}{d\tau}Tr(F'(H))\big)^2.
	\end{equation*}
	Similarly we prove the inequalities for powers of $F$. If $F''>0$ and $G(\lambda)\leq a$  for some $a>0$ then
		\begin{equation*}
	(an)^{-1}\frac{d^2}{d\tau^2} Tr(F(H))\geq \big(\frac{d}{d\tau}Tr(F'(H))\big)^2,
	\end{equation*}
	i.e.
	\begin{equation*}
		\frac{d^2}{d\tau^2} \exp(-\frac{Tr(F(H))}{an}\leq 0
	\end{equation*}
	proving the claim. Similarly we prove the inequalities for powers of $F$.
\end{proof}
\begin{rem}
	When $H=(1-\tau)A+\tau B$ and $F''\geq 0$ then the first statement of theorem \ref{Theorem-trace-class-inequalities} that $\phi_F(\tau):=Tr(F((1-\tau)A+\tau B))$ is convex implies Klein's inequality (see e.g.\cite{Car2010},\cite{Petz1994})
	\begin{equation}\label{eq-Kleins-inequality}
		Tr(F(B)-F(A)-F'(A)(B-A))\geq 0
	\end{equation}
	since by convexity $\dot{\phi}_F(0)\leq \tau^{-1}(\phi_F(\tau)-\phi_F(0)) $ for all $0<\tau\leq 1$.
	We also note the following chain of inequalities for convex functions $\phi_F(\tau)$:
		\begin{equation}\label{eq-two-sided-convex-trace-inequality}
		0\leq (1-\tau)\phi_F(0)+\tau \phi_F(1)-\phi_F(\tau) \leq\tau(1-\tau)Tr((F'(B)-F'(A))(B-A)).
	\end{equation}
\end{rem}
\begin{rem}
	The logarithmic convexity of $Tr(e^H)$, $H=(1-\tau)A+\tau B$ is also known as the Peierls-Bogoliubov inequality (\cite{Car2010}).  The convexity of $Tr(H^p)^{1/p}$ implies Minkowski's inequality for Schatten-p-norms of (here positive hermitian) matrices, the concavity of $Tr(H^p)^{1/p}$ implies a reversed Minkowski-inequality. The concavity of the determinant in the last statement of the corollary can be found in \cite{BB1961}. The log-concavity can be seen by applying the arithmetic-geometric-mean inequality in eq.\eqref{concavity-of-determinants}.
\end{rem}
\begin{rem}
	The conclusions of corollary \ref{cor-trace-convexity-inequalities-1} with exception of point (5) also hold for self-adjoint operators on infinite dimensional Hilbert spaces provided all traces are finite. See also .... for more details.
\end{rem}
In the following corollary we make a more precise use of the expression $Tr(G(H))$.
\begin{cor}\label{cor-trace-convexity-inequalities-2}
	Under the conditions of theorem \ref{Theorem-trace-class-inequalities} suppose that $\ddot{H}=0$. Let $F$ be twice differentiable with twice differentiable inverse $F^{-1}$ and define
	\begin{equation}\label{def-A-function-trace-convexity}
		A(y):=-\,\frac{(F^{-1})'(y)}{(F^{-1})''(y)}.
	\end{equation} 
The following properties hold:
	\begin{enumerate}
		\item Let $F'(\lambda)>0$ and $F''(\lambda)>0$. If $A(y)$ is a concave function then $F^{-1}(\frac{Tr(F(H))}{n})$ is convex.
		\item Let $F'(\lambda)<0$ and $F''(\lambda)>0$. If $A(y)$ is a concave function then $F^{-1}(\frac{Tr(F(H))}{n})$ is concave.
		\item Let $F'(\lambda)>0$ and $F''(\lambda)<0$. If $A(y)$ is a concave function then $F^{-1}(\frac{Tr(F(H))}{n})$ is concave.
		\item Let $F'(\lambda)<0$ and $F''(\lambda)<0$. If $A(y)$ is a convex function then $F^{-1}(\frac{Tr(F(H))}{n})$ is  convex.
	\end{enumerate}
\end{cor}
\begin{proof}
	By the definition of $A(y)$ we obviously have $A(F(\lambda))=G(\lambda)=\frac{F'(\lambda)^2}{F''(\lambda)}$. Let $\phi(\tau):=\frac{Tr(F(H))}{n}$. We compute
	\begin{equation*}
		\frac{d^2}{d\tau^2} F^{-1}(\phi(\tau))=(F^{-1})''(F(\phi(\tau)))\cdot \dot{\phi}(\tau)^2+(F^{-1})'(\phi(\tau))\cdot \ddot{\phi}(\tau).
	\end{equation*}
	Let $F'(\lambda)>0$ and $F''(\lambda)>0$. By eq. \eqref{eq-trace-convexity-concavity} of theorem \ref{Theorem-trace-class-inequalities} and the assumption that $\ddot{H}=0$ we have
	\begin{equation*}
		\ddot{\phi}(\tau)\geq \frac{1}{n}\sum_{j=1}^nF''(\lambda_j)\dot{\lambda_j}^2
	\end{equation*}
	and therefore as in the proof of corollary\ref{cor-trace-convexity-inequalities-1} by the Cauchy-Schwarz inequality and the definition of $A(y)$:
	\begin{equation*}
		\ddot{\phi}(\tau)\geq \frac{\dot{\phi}(\tau)^2}{\frac{1}{n}\sum_{j=1}^nA(F(\lambda_j))}.
	\end{equation*}
	By assumption $A(y)$ is concave therefore by Jensen's inequality
	\begin{equation*}
		\frac{1}{n}\sum_{j=1}^nA(F(\lambda_j))\leq A(\frac{1}{n}\sum_{j=1}^nF(\lambda_j))=A(\phi(\tau)).
	\end{equation*}
	Since $F'(\lambda)>0$ we also have $(F^{-1})'(y)>0$ and therefore we get the differential inequality
	\begin{equation*}
		\frac{d^2}{d\tau^2} F^{-1}(\phi(\tau))\geq(F^{-1})''(F(\phi(\tau)))\cdot \dot{\phi}(\tau)^2+(F^{-1})'(\phi(\tau))\cdot \frac{\dot{\phi}(\tau)^2}{A(\phi(\tau))} .
	\end{equation*}
	By the definition of $A(y)$ the r.h.s is equal to zero proving the claim. Similarly we prove the remaining three inequalities.
\end{proof}
\begin{rem}
	When at least one $\dot{\lambda}_j\neq 0$ the inequalities are even strict which follows from the strict concavity/convexity of $F$ and $F^{-1}$.
\end{rem}
\begin{rem}
	The functions $\lambda^p, e^{\lambda}, \ln\lambda$ discussed before are also covered by the corollary. Another example is $F(\lambda)=-\lambda\ln \lambda +\lambda$, $0<\lambda\leq 1$ applied to density matrices, i.e. positive hermitian matrices with trace equal to $1$.  $F(\lambda)$ is strictly concave on $]0,1]$ and because of the linear term also strictly increasing. The quantity $S(H):=Tr(-H\ln H)$ is called the von Neumann entropy of a density matrix $H$. As before let $\phi_F(\tau)=\frac{Tr(F((1-\tau)A+\tau B))}{n}$ where $A$and $B$ are density matrices. $\phi_F(\tau)$ is strictly concave and Klein's inequality \eqref{eq-Kleins-inequality} reads as follows
	\begin{equation*}
		Tr(B\ln A) \leq Tr(B\ln B)=-S(B)
	\end{equation*}
	Since $\phi_F(\tau)$ is concave we also have
	\begin{equation*}
	\phi_F(\tau)\leq (1-\tau)\phi_F(0)+\tau\phi_F(\tau)+\tau(1-\tau)\big(\dot{\phi}_F(0)-\dot{\phi}_F(1)\big)
	\end{equation*}
	which is equivalent to
	\begin{equation*}
		S((1-\tau)A+\tau B))\leq (1-\tau)S(A)+\tau S(B)+\tau(1-\tau)Tr((\ln B-\ln A)(B-A)),
	\end{equation*}
	i.e.
	\begin{equation*}
		S((1-\tau)A+\tau B))\leq (1-\tau)^2S(A)+\tau^2S(B)-\tau(1-\tau)Tr(A\ln B+B\ln A).
	\end{equation*}
	The inverse of $F$ is given by
	\begin{equation*}
		F^{-1}(y)=e^{1+W_{-1}(-y/e)}
	\end{equation*}
	where $W_{-1}$ denotes the negative branch of the Lambert-W function $W(y)$ which is defined by $y=W(y) e^{W(y)}$ and $W_{-1}<-1$ (see e.g.\cite{CGHJK1996} ). $F^{-1}$ is strictly increasing and strictly convex. It is easy to show that $A(y)$ defined in \eqref{def-A-function-trace-convexity} is (strictly) concave. Therefore the function $\psi(t)$ defined by
	\begin{equation}
	\psi(t):=F^{-1}(\frac{Tr(F((1-\tau)A+\tau B))}{n})=F^{-1}(\frac{1+S((1-\tau)A+\tau B)}{n})
	\end{equation}
	is strictly concave in $\tau$. In particular $\dot{	\psi}(0)\geq 	\psi(1)-\psi(0)$, that is
	\begin{equation*}
		(F^{-1})'(\frac{Tr(F(A))}{n})\frac{Tr(F'(A)(B-A))}{n}\geq F^{-1}(\frac{Tr(F(B)))}{n})-F^{-1}(\frac{Tr(F(A))}{n})
	\end{equation*}
	which improves Klein's inequality since $F^{-1}$ is convex. Also note that $(F^{-1})'(y)=\frac{1}{-\ln F^{-1}(y)}$.
\end{rem}

\subsection{Inequalities for the bottom of the spectrum.}
Here we extend the conclusions of corollary \ref{cor-trace-convexity-inequalities-2} under slightly more restrictive conditions on $F$ for finite sums over the bottom of the spectrum of the operators $H(\tau)$.

\begin{cor}\label{cor-trace-convexity-inequalities-3-bottom-of-spectrum}
	Let $H=H_{\tau}$, $\tau$ real,  be an analytic one parameter family of self-adjoint operators on a  Hilbert space $(\mathcal{H},\langle,\rangle)$ satisfying the conditions of theorem \ref{thm-second order Feynman-Hellmann theorem}.
	Let $m$ be a positive integer. Suppose that $\lambda_1(\tau)<\ldots<\lambda_{m+1}(\tau)$ are isolated eigenvalues of the bottom of the spectrum. Let $F$ be a twice differentiable function defined on an interval containing the spectra of $H_{\tau}$ and let $f=F'$ be the derivative function of $F$. Then the following assertions hold.
	\begin{enumerate}
		\item  Suppose that $\ddot{H}\leq 0$. If $F'\leq 0$, $F''>0$ then
		\begin{equation}\label{eq-1-cor-trace-convexity-inequalities-3-bottom-of-spectrum}
			\frac{d^2}{d\tau^2} \sum_{j=1}^m F(\lambda_j)\geq \bigg(  \sum_{j=1}^m F'(\lambda_j)\dot{\lambda}_j\bigg)^2/ \sum_{j=1}^m \frac{ F'^2(\lambda_j)}{ F''(\lambda_j)}
		\end{equation}
		In particular, if $F'<0$ and the function $A(y)$ defined in eq.\eqref{def-A-function-trace-convexity} of corollary \ref{cor-trace-convexity-inequalities-2} is concave then
		\begin{equation*}
			F^{-1}(\frac{1}{m}\sum_{j=1}^m F(\lambda_j)) 
		\end{equation*}
		is a concave function of $\tau$.  
		\item Suppose that $\ddot{H}\leq 0$. If $F'\geq 0$, $F''<0$ then
		\begin{equation}\label{eq-2-cor-trace-convexity-inequalities-3-bottom-of-spectrum}
			\frac{d^2}{d\tau^2} \sum_{j=1}^m F(\lambda_j)\leq \bigg(  \sum_{j=1}^m F'(\lambda_j)\dot{\lambda}_j\bigg)^2/ \sum_{j=1}^m \frac{ F'^2(\lambda_j)}{ F''(\lambda_j)}
		\end{equation}
		In particular, if $F'>0$ and the function $A(y)$ defined in eq.\eqref{def-A-function-trace-convexity} of corollary \ref{cor-trace-convexity-inequalities-2} is concave then
		\begin{equation*}
			F^{-1}(\frac{1}{m}\sum_{j=1}^m F(\lambda_j)) 
		\end{equation*}
		is a concave function of $\tau$.
		\item If $F$ and $F''$ are concave functions such that $2F'(\lambda_j) +F''(\lambda_j)(\lambda_{m+1}-\lambda_j)\leq 0$ for all $j=1,\ldots m$ then
			\begin{equation}\label{eq-3-cor-trace-convexity-inequalities-3-bottom-of-spectrum}
			\frac{d^2}{d\tau^2} \sum_{j=1}^m F(\lambda_j)\geq  \sum_{j=1}^m F''(\lambda_j)\langle\dot{H}u_j,\dot{H}u_j\rangle+F'(\lambda_j)\langle\ddot{H}u_j,u_j\rangle.
		\end{equation}
	\end{enumerate}
\end{cor}

\begin{proof}
As in the proof of theorem \ref{Theorem-trace-class-inequalities} we compute
\begin{equation*}
	\frac{d^2}{d\tau^2} \sum_{j=1}^m F(\lambda_j)=\sum_{j=1}^mf(\lambda_j)(\ddot{\lambda_j}-\langle \ddot{H}u_j ,u_j\rangle)+\sum_{j=1}^mf(\lambda_j)(\langle \ddot{H}u_j ,u_j\rangle) +F''(\lambda_j)\dot{\lambda_j}^2
\end{equation*}
The first sum we replace by equation \ref{Feynman-Hellmann-sum-rule-general-splittig-with-J} from the proof of our main theorem \ref{thm-second order Feynman-Hellmann theorem}. We obtain
\begin{equation}\label{eq-Feynman-Hellmann-bottom-of-spetrum}
	\begin{split}
		\frac{d^2}{d\tau^2} \sum_{j=1}^m F(\lambda_j)&=\underset{\lambda_j\neq\lambda_k}{\sum_{j=1}^m\sum_{k=1}^m}\frac{f(\lambda_k)-f(\lambda_j)}{\lambda_k-\lambda_j}\big|\langle u_k,\dot{H}u_j\rangle\big|^2 -\sum_{j=1}^m\sum_{k=m+1}^m\frac{2\,f(\lambda_j)}{\lambda_k-\lambda_j}\big|\langle u_k,\dot{H}u_j\rangle\big|^2\\
		&+\sum_{j=1}^mf(\lambda_j)(\langle \ddot{H}u_j ,u_j\rangle) +F''(\lambda_j)\dot{\lambda_j}^2.\\
	\end{split}
\end{equation}
\begin{enumerate}
	\item If $F'=f<0$ and $F''>0$ then $f<0$ is strictly increasing and therefore the first two terms are positive. The third term is positive since $ \ddot{H}\leq 0$. The rest follows from the proofs of corollary \ref{cor-trace-convexity-inequalities-1} and corollary \ref{cor-trace-convexity-inequalities-2}.
	\item The proof in the case $F'=f>0$ and $F''<0$ follows the same lines.
	\item If $F''=f'$ is concave we bound the first term from below as in the proof of our main theorem \ref{thm-second order Feynman-Hellmann theorem} and the assertion follows from the additional hypothesis on $2F'(\lambda_j) +F''(\lambda_j)(\lambda_{m+1}-\lambda_j)$.
\end{enumerate}
\end{proof}

\begin{rem}
	We do not know simple criteria for matrices to have only isolated eigenvalues, one dimensional Schr\"{o}dinger-type operators or regular Sturm-Liouville problems will satisfy this condition of the above corollary, see Section 5.
\end{rem}

\subsection{Eigenvalues of a Hermitian matrix sum.}
We consider the case where $H=(1-\tau)A+\tau B$ is a convex linear combination of two hermitian $n\times n$ matrices $A,B$. In this case we denote the eigenvalues $\lambda_j=\lambda_j(\tau)$ of $H$ non-decreasing in $j$ also by  $\lambda_j((1-\tau)A+\tau B)$ and $\lambda_j(A):=\alpha_j$, $\lambda_j(B):=\beta_j$.

From eq. \eqref{Feynman-Hellmann-sum-rule-constant} we conclude that for any positive integer $m$ the function
\begin{equation}\label{definition-theta-m}
	\theta_m(\tau):=\sum_{j=1}^m\lambda_j((1-\tau)A+\tau B)
\end{equation}
is a concave function of $\tau$. Consequently for all $\tau \in [0,1]$,
\begin{equation}\label{theta-m-concavity}
	0\geq (1-\tau)\theta_m(0)+\tau \theta_m(1) - \theta_m(\tau)\geq \tau(1-\tau)(\dot{\theta}_m(1)-\dot{\theta}_m(0))
\end{equation}
where the second inequality follows from the fact that $\theta_m(\tau)$ is a differentiable concave function. We therefore have the following result. 
\begin{prop}
	Let $A,B$ be two $n\times n$ hermitian matrices with eigenvalues $\alpha_1\leq\ldots\leq \alpha_n$ and $\beta_1\leq\ldots\leq \beta_n$, respectively. For $\tau \in [0,1]$ let $H=(1-\tau)A+\tau B$ with (non-decreasing) eigenvalues denoted by $\lambda_j((1-\tau)A+\tau B)$. then the following inequalities hold
	\begin{equation}\label{A-B-eigenvalue-concavity-bottom-1}
		\sum_{j=1}^m(1-\tau)\alpha_j+\tau \beta_j\leq 	\sum_{j=1}^m\lambda_j((1-\tau)A+\tau B)
	\end{equation}
	and
	\begin{equation}\label{A-B-eigenvalue-concavity-bottom-2}
		\sum_{j=1}^m\lambda_j((1-\tau)A+\tau B)\leq	\sum_{j=1}^m(1-\tau)^2\alpha_j+\tau^2 \beta_j+(1-\tau) \tau(\alpha_{n+1-j}+\beta_{n+1-j}).
	\end{equation}
	In particular,
	\begin{equation}\label{A+B-eigenvalue-bounds}
		\sum_{j=1}^m\alpha_j+\beta_j\leq\sum_{j=1}^m\lambda_j(A+B)\leq	\frac{1}{2}\sum_{j=1}^m\alpha_j+\beta_j+\alpha_{n+1-j}+\beta_{n+1-j}.
	\end{equation}
\end{prop}
\begin{proof}
	The first two inequalities of the proposition follow from the concavity inequalities for $\theta_m$ given in eq. \eqref{theta-m-concavity}. We only have to exploit the expression $\dot{\theta}_m(1)-\dot{\theta}_m(0)$. We get
	\begin{equation*}
		\begin{split}
			\dot{\theta}_m(1)-\dot{\theta}_m(0)&=\sum_{j=1}^m\dot{\lambda}_j(1)-\dot{\lambda}_j(0)\\
			&=\sum_{j=1}^m\langle (B-A)u_j(1),u_j(1)\rangle-\langle (B-A)u_j(0),u_j(0)\rangle\\
			&=\sum_{j=1}^m\alpha_j+\beta_j-\langle A u_j(1),u_j(1)\rangle-\langle B u_j(0),u_j(0)\rangle\\
			&\geq\sum_{j=1}^m\alpha_j+\beta_j-\alpha_{n+1-j}-\beta_{n+1-j}\\
		\end{split}
	\end{equation*}
	where the last line follows from the variational principle. Evaluating the bounds at $\tau=\frac{1}{2}$ we obtain eigenvalue bounds for the sum $A+B$.
\end{proof}
Applying theorem \ref{thm-2nd-order-Feynman-Hellmann-squeezing} to the bottom and the top of the spectrum of hermitian matrices yields the following result.

\begin{prop}
	Let $H=H(\tau)$ be an analytic one parameter family of $n\times n$ hermitian matrices satisfying with eigenvalues $\lambda_j$, $j=1,\ldots n$ and  let $J=\{\lambda_1,\ldots \lambda_m\}$.  Suppose that $\lambda_{m+1}-\lambda_m>0$. Then
	\begin{equation}\label{Feynman-Hellmann-sum-rule-constant-squeezing-matrices-2-sided}
		(\lambda_{n}-\lambda_1)^{-2}\sum_{j=1}^m\langle [\dot{H},[H,\dot{H}]] u_j,u_j\rangle\leq \sum_{j=1}^m \langle \ddot{H}u_j,u_j\rangle	-\ddot{\lambda}_j  
		\leq (\lambda_{m+1}-\lambda_m)^{-2}\sum_{j=1}^m\langle [\dot{H},[H,\dot{H}]] u_j,u_j\rangle
	\end{equation}
	
\end{prop}
We rewrite the quadratic second order Feynman-Hellmann sun rule eq. \eqref{Feynman-Hellmann-sum-rule-quadratic} applying the chain rule or product rule for operator differentiation as follows:
\begin{prop}
	Let $H=H(\tau)$ be an analytic one parameter family of $n\times n$ hermitian matrices with eigenvalues $\lambda_j$, $j=1,\ldots n$. 
	for all $z=z(\tau)$
	\begin{equation}\label{Feynman-Hellmann-sum-rule-quadratic-operator-derivative-version}
		\begin{split}
			&\frac{1}{3}\sum_{\lambda_j\in J} \frac{d^2}{d\tau^2}(z-\lambda_j)^3-\langle  \frac{d^2}{d\tau^2}(z-H)^3u_j,u_j\rangle-\langle [\dot{H},[H,\dot{H}]]u_j,u_j\rangle\\
			&=\sum_{\lambda_j\in J} (z-\lambda_j)^2(\langle \ddot{H}u_j,u_j\rangle-\ddot{\lambda}_j) -2(z-\lambda_j)(\langle \dot{H}u_j,\dot{H}u_j\rangle-\dot{\lambda}_j^2)\\
			&=2\sum_{\lambda_j\in J}\sum_{\lambda_k\notin J}\frac{(z-\lambda_j)(z-\lambda_k)}{\lambda_k-\lambda_j}\big|\langle \dot{H}u_j,u_k\rangle\big|^2.\\
		\end{split}
	\end{equation}
	
\end{prop}
Another consequence of quadratic second order Feynman-Hellmann sun rule eq. \eqref{Feynman-Hellmann-sum-rule-quadratic} is the following
\begin{prop}
	Let $H=H(\tau)$ be an analytic one parameter family of $n\times n$ hermitian matrices with eigenvalues $\lambda_j$, $j=1,\ldots n$ and suppose that $\ddot{H}\leq 0$. For $z$ fixed let $J=\{\lambda_j: \lambda_j<z\}$ and suppose that no eigenvalues leave or enter $J$. Then the function
	\begin{equation}
		\tau \mapsto\bigg(\sum_{\lambda_j\in J}(z-\lambda_j)^3\bigg)^{1/3}
	\end{equation}
	is convex.
\end{prop}

\section{Sum rules and Feynman-Hellmann theorems for Zeros of Bessel functions}
We consider the Dirichlet eigenvalue problem
\begin{equation}\label{eq-Bessel-eigenvalue-eq}
	-u''(x)+\frac{\nu^2-1/4}{x^2}u(x)=Eu(x), \quad u(0)=0, u(1)=0
\end{equation}
Its eigenvalues are given by $E_{k}=E_{k}(\nu)=j^2_{\nu,k}$ where $j_{\nu,k}$ denotes the $k-$th zero of the Bessel function $J_{\nu}(x)$. The corresponding normalized eigenfunctions are given by
\begin{equation}\label{eq-Bessel-eigenfunction}
	u_k(x)=\frac{2^{1/2}}{|J_{\nu+1}(j_{\nu,k})|}\,\big(j_{\nu,k}x\big)^{1/2}J_{\nu}(j_{\nu,k}x)
\end{equation}
though we shall not make explicit use of these expressions. In the following we summarize known results about the zeros of $J_{\nu}(x)$: For each $k$, $j^2_{\nu,k}/\nu^2$ is decreasing in $\nu\in ]0,\infty[$ and $j_{\nu,k}/\nu$ is decreasing in $\nu\in ]0,\infty[$ and $j^2_{\nu,k}/\nu$ is increasing in $\nu\in ]\nu_0,\infty[$ when $\nu_0$ is sufficiently large. These properties can be proven either applying the classical (first order) Feynman-Hellmann theorem eq.\eqref{Feynman-Hellmann theorem-1} or integral representations of Bessel functions, see \cite{IsmZha1988}, theorem 3.1, \cite{LewMul1977}, theorem 3.1. In \cite{LewMul1977}, theorem 3.1, it is also shown that $\dot{E}_k=2j_{\nu,k}\dot{j}_{\nu,k}$ increases for $\nu\in ]\nu_k,\infty[$ for some $\nu_k$  sufficiently large using an integral representation of $J_{\nu}(x)$. Here we should also mention the elementary inequality
\begin{equation}
\dot{E}_k\geq 2\nu
\end{equation} 
all positive integers $k$ which follows from the first-order Feynman-Hellman theorem and the fact that $0\leq x\leq 1$. As a consequence for all $0\leq \nu_1\leq \nu_2$ and all positive integers $k$ we have
\begin{equation}
	E_k(\nu_2)-\nu_2^2\geq E_k(\nu_1)-\nu_1^2
\end{equation} 
and therefore letting $\nu_1$ tend to zero
\begin{equation}
	E_k(\nu_2)\geq\nu_2^2+E_k(0)\geq \nu_2^2.
\end{equation} 
These inequalities are asymptotically sharp for each $k$ since $\nu^{-2}E_k(\nu)\to 1$ as $\nu$ tends to infinity (see e.g. \cite{QuWon1999} with sharp lower and upper bounds for $j_{\nu,k}$).

A sum rule derived for the zeros of Bessel functions was first derived by Arai \cite{Arai1992} which is a particular case of the quadratic sum rule eq.\eqref{HS-sumrule-discrete} corresponding to the coefficient of $z^2$ (the TRK sum rule), $G$ a multiplication operator and $J$ consisting of a single eigenvalue of the spectrum and relates integrals of Bessel functions and its zeros.

The (full) quadratic sum rule \eqref{HS-sumrule-discrete} with the coordinate function $G(x)=x$ leads as in section 2.3 to the inequality

\begin{equation}\label{eq-HS-sum-rule-Bessel-ineq-1}
		\sum_{k}(z-E_k)_{+}^2\leq 4\sum_{k}(z-E_k)_{+}(E_k-\frac{\nu^2-1/4}{2\nu}\dot{E}_k)
\end{equation}
and we conclude that for $\nu>1/2$:
\begin{equation}\label{eq-HS-sum-rule-Bessel-ineq-2}
	\frac{d}{d\nu}(\nu^2-1/4)^{-5/2}\sum_{k}(z(\nu^2-1/4)-E_k)_{+}^2\geq 0.
\end{equation}

From theorem \ref{thm-2nd-order-Feynman-Hellmann-squeezing} we obtain for any positive integer $m$ the concavity inequality

\begin{equation}\label{eq-Feynman-Hellmann-Bessel-ineq-concavity-1}
\sum_{k=1}^m\ddot{E}_k-\nu^{-1}\dot{E}_k\leq 0.
\end{equation}
In particular $\displaystyle \sum_{k=1}^m\nu^{-1}\dot{E}_k=2\sum_{k=1}^m\nu^{-1}j_{\nu,k}\dot{j}_{\nu,k}$ is decreasing (in $\nu$). Integrating inequality \eqref{eq-Feynman-Hellmann-Bessel-ineq-concavity-1} we obtain the following proposition.
\begin{prop}
	For all $\nu_2>\nu>\nu_1>0$ and any positive integer $m$:
\begin{equation}\label{eq-Feynman-Hellmann-Bessel-ineq-concavity-2}
	\sum_{k=1}^mE_k(\nu_2)-E_k(\nu)\leq \frac{\nu_2^2-\nu^2}{\nu^2-\nu_1^2}\sum_{k=1}^mE_k(\nu)-E_k(\nu_1).
\end{equation}
	
\end{prop}

\begin{rem}
	Since for any $k$ fixed $\lim_{\nu\to\infty}\nu^{-1}j_{\nu,k}=1$ (see e.g. \cite{QuWon1999} even with sharp upper and lower bounds), i.e. $\lim_{\nu\to\infty}\nu^{-2}E_{\nu,k}=1$, we get from the proposition that for any positive integer $m$ and all $\nu>\nu_1>0$
	\begin{equation}\label{eq-Feynman-Hellmann-Bessel-ineq-concavity-3}
		m(\nu^2-\nu_1^2)\leq \sum_{k=1}^mE_k(\nu)-E_k(\nu_1).
	\end{equation}
Alternatively we may rewrite inequality \eqref{eq-Feynman-Hellmann-Bessel-ineq-concavity-1} as follows:

\begin{equation}\label{eq-Feynman-Hellmann-Bessel-ineq-concavity-4}
	\sum_{k=1}^m\frac{d^2}{d\nu^2}\frac{E_k}{\nu^2}+3\nu^{-1}\frac{d}{d\nu}\frac{E_k}{\nu^2}\leq 0.
\end{equation}
Since $\frac{E_k}{\nu^2}$ is decreasing in $\nu$, even strictly decreasing, we obtain the inequality
\begin{equation}
	\frac{\sum_{k=1}^m\nu^{-2}E_k(\nu)-\nu_2^{-2}E_k(\nu_2)}{\sum_{k=1}^m\nu_1^{-2}E_k(\nu_1)-\nu^{-2}E_k(\nu)}\geq 
	\frac{\nu^{-2}-\nu_2^{-2}}{\nu_1^{-2}-\nu^{-2}}
\end{equation}
for any positive integer $m$ and all $\nu_2>\nu>\nu_1>0$.
\end{rem}
As shown in \cite{HarStu2010} using integral transforms of $(z-E)_{+}^2$ the conclusions of theorem \ref{thm-Feynman-Hellmann theorem in quadratic sum rules} can ve easily extended to a large class of functions as, for example the partition function. 
\begin{prop}
	For $\nu>\frac{1}{2}$ and all $t>0$ the function
	\begin{equation}\label{eq-partition-function-square-of-Bessel-zeros}
		H(t,\nu):= (\nu^2-\frac{1}{4})^{-\frac{1}{2}}\sum_ke^{-t(\nu^2-\frac{1}{4})^{-\frac{1}{2}}E_k}
	\end{equation}
	is increasing in $\nu$. Consequently, by Laplace transforming eq. \eqref{eq-partition-function-square-of-Bessel-zeros}  we also have for all $p>\frac{1}{2}$:
	\begin{equation} \label{eq-negative-moments-square-of-Bessel-zeros}
			\frac{d}{d\nu}(\nu^2-1/4)^{p-1/2}\sum_{k}E_k^{-p}\geq 0.
	\end{equation}
\end{prop}
\begin{rem}
	A similar statement about the monotonicity of $H(t,\nu)$ holds for $0\leq \nu<\frac{1}{2}$.  
\end{rem}
\begin{rem}
For $p=1,2,3$ the series are explicitly known, e.q. $\displaystyle  \sum_{k}E_k^{-1} =\frac{1}{4(\nu+1)}$ (\cite{Cal1977}).
In these particular cases the derivative in eq. \eqref{eq-negative-moments-square-of-Bessel-zeros} are strictly positive.
\end{rem}
Applying the abstract second order Feynman-Hellmann sumrule eq. \ref{Feynman-Hellmann-sum-rule-quadratic} and using
$\dot{H}=2\nu x^{-2}$, $\ddot{H}=\nu^{-1}\dot{H}$ we get for all positive integers $n$:
\begin{equation}\label{eq-Bessel-ev-Feynman-Hellmann-sum-rule-quadratic}
	\begin{split}
		&\sum_{k=1}^n (z-E_k)^2(\nu^{-1}\dot{E}_k-\ddot{E}_k) -2(z-\lambda_j)(\langle \dot{H}u_k,\dot{H}u_k\rangle-\dot{E}_k^2)\\
		&=2\sum_{k=1}^n\sum_{l=n+1}^{\infty}\frac{(z-E_k)(z-E_l)}{E_l-E_k}\big|\langle \dot{H}u_k,u_l\rangle\big|^2.\\
	\end{split}
\end{equation}
By the following lemma we see that the expression $\langle \dot{H}u_k,\dot{H}u_k\rangle$ is well-defined only if $\nu>1$ and can be expressed in terms of $E_k$ and $\dot{E}_k$:
\begin{lem}
	Let $\nu>1$. For all positive integers $k$:
	\begin{equation}\label{Bessel-H-dot-square-expectation}
	\langle \dot{H}u_k,\dot{H}u_k\rangle=\frac{2\nu^2E_k}{\nu^2-1}	\big(1+\frac{\dot{E}_k}{2\nu}\big).
	\end{equation}
\end{lem}
\begin{proof}
	The proof is elementary. We multiply the Schr\"{o}dinger equation $H u_k=E_ku_k$ successively by $x^{-2}u_k$ and $x^{-1}u'_k$, integrate the resulting equation which requires $\nu>1$ in order to get finite terms and isolate the expression $\displaystyle \int_{0}^{1}x^{-4}u^2_k\,dx$. Finally we note that
		\begin{equation}
		u_k'(1)^2=2E_k
	\end{equation}
	which follows from multiplying the  Schr\"{o}dinger equation $H u_k=E_ku_k$ by $xu'_k$ and integrating or alternatively from the explicit solutions given in eq. \eqref{eq-Bessel-eigenfunction}.
\end{proof}
We shall also make use of a property of functions satisfying a first order linear partial differential inequality:
\begin{lem}
Let $a(\nu),b(\nu)$ be continuous with primitives $A(\nu)$ and $B(\nu)$,respectively. Suppose that $G(\nu,z)$ satisfies the following partial differential inequality
	\begin{equation}
		\frac{\partial\, G(\nu,z)}{\partial\nu}-a(\nu)G(\nu,z)+z\,b(\nu)\frac{\partial\, G(\nu,z)}{\partial z}\leq 0.
	\end{equation}
	Then for any $z$ fixed
	\begin{equation}
		\frac{d\,}{d\nu}\, e^{-A(\nu)}G(\nu,e^{B(\nu)}z)\leq 0.
	\end{equation}
\end{lem}
\begin{proof}
	The proof is an easy computation.
\end{proof}
Since $\displaystyle 1\leq\frac{\dot{E}_k}{2\nu}$ we have
\begin{equation}\label{Bessel-H-dot-square-expectation-ineq}
	\langle \dot{H}u_k,\dot{H}u_k\rangle\leq\frac{2\nu\,E_k\dot{E}_k}{\nu^2-1}.
\end{equation}
Inserting this inequality into the sum rule\eqref{eq-Bessel-ev-Feynman-Hellmann-sum-rule-quadratic} when $E_n\leq z\leq E_{n+1}$ we get the following proposition:
\begin{prop}
	For $\nu>1$ let
\begin{equation}
	F(\nu, z):=\frac{1}{3}\sum_k (z-E_k)_{+}^3,\quad 	G(\nu, z):=\frac{\partial\, F(\nu,z)}{\partial\nu}=-\sum_k (z-E_k)_{+}^2\dot{E}_k.
\end{equation}
Then
\begin{equation}
\frac{\partial^2\, F(\nu,z)}{\partial\nu^2}-\big(\frac{1}{\nu}+\frac{4}{\nu^2-1}\big)\frac{\partial\, F(\nu,z)}{\partial\nu}+\frac{4\nu\,z}{\nu^2-1}\frac{\partial^2\, F(\nu,z)}{\partial z\partial\nu}\leq 0.
\end{equation}
In particular, 
\begin{equation}
	\nu^{-1} (\nu^2-1)^{-2}G(\nu,  (\nu^2-1)z)=-\nu^{-1}\sum_k \big(z-(\nu^2-1)^{-1}E_k\big)_{+}^2\dot{E}_k
\end{equation}
is decreasing in $\nu$.
\end{prop}
We conclude this section with a convexity result for (negative) moments of $E_k$:
\begin{prop}
	Let $p>\frac{1}{2}$. Then for all $\nu>1$:
	\begin{equation}\label{Bessel-negative-moments-second order-diff-ineq}
		\frac{d}{d\nu}\,\nu^{-1}(\nu^2-1)^{p+1}\sum_{k=1}^{\infty}\frac{dE_k^{-p}}{d\nu}\leq 0
	\end{equation}
\end{prop}
\begin{proof}
Applying theorem \ref{thm-second order Feynman-Hellmann theorem} to the function $f(E)=E^{-1-p}$ and the full spectrum and using the identity \eqref{Bessel-H-dot-square-expectation} we have the following inequality:
\begin{equation}
	\sum_{k=1}^{\infty}E_k^{-p-1}(\ddot{E}_k-\nu^{-1}\dot{E}_k)+(p+1)E_k^{-p-2}\big(\frac{2\nu^2E_k}{\nu^2-1}	\big(1+\frac{\dot{E}_k}{2\nu}\big)-\dot{E}^2_k)\big)\geq 0
\end{equation}
Reorganizing terms and using as before the inequality $1\leq (2\nu)^{-1}\dot{E}_k$ we find
\begin{equation*}
	\sum_{k=1}^{\infty}\frac{d^2E_k^{-p}}{d\nu^2}-\big(\frac{1}{\nu}-\frac{2(p+1)\nu}{\nu^2-1}\big)\frac{dE_k^{-p}}{d\nu}\leq 0
\end{equation*}
proving the assertion.
\end{proof}
\section*{Acknowledgments}
The author thanks Evans Harrell for useful discussions on the concepts presented in this work.

\bibliographystyle{amsplain}

\end{document}